\documentclass[a4paper, reqno, twoside]{amsart}
\pdfoutput=1
\usepackage{etex}
\usepackage{amsfonts,amssymb,graphics,amsthm,amsmath}
\usepackage{latexsym}
\usepackage[2emode]{psfrag}
\usepackage{mathrsfs}
\usepackage[all]{xy}
\usepackage{enumerate}
\usepackage[utf8]{inputenc}
\usepackage{lscape}
\usepackage{a4wide}
\usepackage[bookmarks=false]{hyperref}
\usepackage{float}

\usepackage{footmisc}
\usepackage{cleveref}

\RequirePackage[bibstyle=alphabetic,
    citestyle=alphabetic,
    giveninits=true]{biblatex}
    
\AtEveryBibitem{%
  \clearfield{note}%
}
\renewbibmacro{in:}{\ifentrytype{article}{}{\printtext{\bibstring{in}\intitlepunct}}}

\usepackage{txfonts}

\usepackage[all]{pstricks}
\usepackage{pst-text,pst-node,pst-coil}
\usepackage{tikz}
\usepackage{mathtools}

\usepackage{cancel}
\usepackage[normalem]{ulem}
\usetikzlibrary{arrows,shapes,backgrounds}

\usepgflibrary{arrows}
\usetikzlibrary{topaths}
\usetikzlibrary{er}

\usepackage{multicol}
\usepackage{tikz-cd}

\newtheorem{theorem}{\sc Theorem}[section]
\newtheorem{proposition}[theorem]{\sc Proposition}

\newtheorem{lemma}[theorem]{\sc Lemma}
\newtheorem{corollary}[theorem]{\sc Corollary}
\theoremstyle{definition}
\newtheorem{definition}[theorem]{\sc Definition}

\newtheorem{example}[theorem]{\sc Example}

\theoremstyle{remark}
\newtheorem{remark}[theorem]{\sc Remark}

\newcommand{\Sf}[1]{\mathsf{#1}}

\newcommand{\bd}[1]{\boldsymbol{#1}}

\newcommand{\I}{\mathbb{I}}
\newcommand{\Z}[1]{\mathcal{Z}(#1)}

\newcommand{\B}[1]{\boldsymbol{#1}}
 \newcommand{\h}{\mathrm{Hom}}
 \newcommand{\id}{\mathrm{Id}}

\newcommand{\Go}{\cG_{\Sscript{0}}}
\newcommand{\Ga}{\cG_{\Sscript{1}}}

\newcommand{\Ad}[1]{{\bd{ad}}_{\Sscript{g}}}

\newcommand{\bB}{\mathscr{B}}
\newcommand{\cC}{\mathscr{C}}

\newcommand{\gG}{\mathscr{G}}

\newcommand{\lL}{\mathscr{L}}

\newcommand{\cA}{{\mathcal A}}
\newcommand{\cB}{{\mathcal B}}
\newcommand{\cD}{{\mathcal D}}

\newcommand{\cG}{{\mathcal G}}

\newcommand{\cI}{{\mathcal I}}

\newcommand{\cX}{{\mathcal X}}
\newcommand{\cY}{{\mathcal Y}}

\newcommand{\Sscript}[1]{\scriptscriptstyle{#1}}

\usepackage{braket, upgreek}
\usepackage{tcolorbox}

\DeclareMathOperator{\Id}{Id}
\DeclareMathOperator{\pr}{pr}
\DeclareMathOperator{\Orb}{Orb}

\DeclareMathOperator{\rep}{rep}

\DeclareMathOperator{\Ske}{Sk}
\DeclareMathOperator{\dCat}{d}
\DeclareMathOperator{\ndCat}{nd}

\DeclareRobustCommand{\Cat}[1]{\mathcal{#1}}

\DeclareRobustCommand{\Rset}[1]{\emph{Sets}\text{-}{#1}}
\DeclareRobustCommand{\rset}[1]{\emph{sets}\text{-}{#1}}

\DeclareRobustCommand{\Uplus}{\mathbin{\scalebox{1.50}{\ensuremath{\uplus}}}}

\DeclareRobustCommand{\fpro}[1]{\underset{#1}{\times}}

\DeclareRobustCommand{\Xob}{\mathcal{X}\sped{0}}
\DeclareRobustCommand{\Xmo}{\mathcal{X}\sped{1}}
\DeclareRobustCommand{\Xmoi}{\mathcal{X}\sped{i}}
\DeclareRobustCommand{\Xmod}{\mathcal{X}\sped{2}}
\DeclareRobustCommand{\Xmot}{\mathcal{X}\sped{3}}
\DeclareRobustCommand{\Xset}[1]{\left(\mathcal{X},{#1}\right)}

\DeclareRobustCommand{\Yob}{\mathcal{Y}\sped{0}}
\DeclareRobustCommand{\Ymo}{\mathcal{Y}\sped{1}}
\DeclareRobustCommand{\Ymoi}{\mathcal{Y}\sped{i}}
\DeclareRobustCommand{\Ymod}{\mathcal{Y}\sped{2}}

\DeclareRobustCommand{\Yset}[1]{\left(\mathcal{Y},{#1}\right)}
\DeclareRobustCommand{\CRset}[1]{\mathsf{CSets}\text{-}{#1}}
\DeclareRobustCommand{\fCRset}[1]{\mathsf{csets}\text{-}{#1}}
\DeclareRobustCommand{\SInc}[1]{\mathcal{I}\text{-}{#1}}

\DeclareRobustCommand{\Aob}{\mathcal{A}\sped{0}}
\DeclareRobustCommand{\Amo}{\mathcal{A}\sped{1}}
\DeclareRobustCommand{\Amoi}{\mathcal{A}\sped{i}}
\DeclareRobustCommand{\Aset}[1]{\left(\mathcal{A},{#1}\right)}

\DeclareRobustCommand{\Bob}{\mathcal{B}\sped{0}}
\DeclareRobustCommand{\Bmo}{\mathcal{B}\sped{1}}

\DeclareRobustCommand{\Bset}[1]{\left(\mathcal{B},{#1}\right)}

\DeclareRobustCommand{\wea}{\sim\ped{\mathsf{we}}}

\DeclareRobustCommand{\imp}[1]{\textit{#1}}

\DeclareRobustCommand{\one}{1}

\DeclareRobustCommand{\siL}{\mathscr{L}_{\textsf{C}}}
\DeclareRobustCommand{\siB}{\mathscr{B}_{\textsf{C}}}

\DeclareRobustCommand{\siIL}{\mathscr{L}_{\mathcal{I}}}

\DeclareRobustCommand{\siIB}{\mathscr{B}_{\mathcal{I}}}

\DeclareRobustCommand{\id}[1]{\Id_{#1}}

\RequirePackage{mathrsfs}

\DeclareRobustCommand{\ssstyle}[1]{{#1}}
\DeclareRobustCommand{\pdue}[3]{ \, {{}_{ {#2} } {#1}_{ {#3} }\,}}
\DeclareRobustCommand{\sped}[1]{_{ {#1} }}

\DeclareRobustCommand{\fpro}[1]{\underset{#1}{\times}}

\DeclareMathOperator{\Grps}{\normalfont{\textbf{Grpd}}}
\DeclareRobustCommand{\riG}{\normalfont{\textbf{Rig}}}
\DeclareRobustCommand{\cring}{\normalfont{\textbf{CRing}}}

\DeclareRobustCommand{\Cat}[1]{\mathcal{#1}}

\DeclareRobustCommand{\Rset}[1]{\mathsf{Sets}\text{-} \mathcal{#1}}
\DeclareRobustCommand{\rset}[1]{\mathsf{Sets}\text{-} {#1}}

\DeclareRobustCommand{\Uplus}{\mathbin{\scalebox{1.50}{\ensuremath{\uplus}}}}

\DeclareRobustCommand{\Xob}{\mathcal{X}\sped{0}}
\DeclareRobustCommand{\Xmo}{\mathcal{X}\sped{1}}
\DeclareRobustCommand{\Xmoi}{\mathcal{X}\sped{i}}
\DeclareRobustCommand{\Xmod}{\mathcal{X}\sped{2}}
\DeclareRobustCommand{\Xmot}{\mathcal{X}\sped{3}}
\DeclareRobustCommand{\Xset}[1]{\left(\mathcal{X},{#1}\right)}

\DeclareRobustCommand{\Yob}{\mathcal{Y}\sped{0}}
\DeclareRobustCommand{\Ymo}{\mathcal{Y}\sped{1}}
\DeclareRobustCommand{\Ymoi}{\mathcal{Y}\sped{i}}
\DeclareRobustCommand{\Ymod}{\mathcal{Y}\sped{2}}

\DeclareRobustCommand{\Yset}[1]{\left(\mathcal{Y},{#1}\right)}

\DeclareRobustCommand{\SInc}[1]{\mathcal{I}\text{-} \mathcal{#1}}

\DeclareRobustCommand{\Aob}{\mathcal{A}\sped{0}}
\DeclareRobustCommand{\Amo}{\mathcal{A}\sped{1}}
\DeclareRobustCommand{\Amoi}{\mathcal{A}\sped{i}}
\DeclareRobustCommand{\Aset}[1]{\left(\mathcal{A},{#1}\right)}

\DeclareRobustCommand{\Bob}{\mathcal{B}\sped{0}}
\DeclareRobustCommand{\Bmo}{\mathcal{B}\sped{1}}

\DeclareRobustCommand{\Bset}[1]{\left(\mathcal{B},{#1}\right)}

\DeclareRobustCommand{\wea}{\sim_{\mathsf{we}}}

\DeclareRobustCommand{\imp}[1]{\textit{#1}}

\DeclareRobustCommand{\ped}[1]{_{#1}}

\DeclareRobustCommand{\one}{1}

\DeclareRobustCommand{\sqre}{\underline{\mathsf{sq}}}

\DeclareRobustCommand{\Sqre}{\underline{\mathsf{Sq}}}

\DeclareRobustCommand{\SqOrb}[1]{\Orb_{\underline{\mathsf{Sq}}}\left({#1}\right)}

\DeclareRobustCommand{\SqRep}[1]{\rep_{\underline{\mathsf{Sq}}}\left({#1}\right)}

\allowdisplaybreaks

\usepackage[a4paper, lmargin=2.3cm, rmargin=2.7cm, bmargin=4.5cm, bindingoffset=0cm]{geometry}

\begin{document}

\title{Categorified groupoid-sets and  their Burnside ring}

\author{Laiachi El Kaoutit}
\address{Universidad de Granada, Departamento de \'{A}lgebra and IEMath-Granada. Facultad de Educaci\'{o}n, Econon\'ia y Tecnolog\'ia de Ceuta. Cortadura del Valle, s/n. E-51001 Ceuta, Spain}
\email{kaoutit@ugr.es}
\urladdr{http://www.ugr.es/~kaoutit/}

\author{Leonardo Spinosa}
\address{University of Ferrara, Department of Mathematics and Computer Science\newline
Via Machiavelli 30, Ferrara, I-44121, Italy}
\email{leonardo.spinosa@unife.it}
\urladdr{https://orcid.org/0000-0003-3220-6479}

\date{\today}

\begin{abstract}
We explore the category of internal categories in the usual category of (right) group-sets, whose objects   are referred  to as \emph{categorified group-sets}.
More precisely, we develop a new Burnside theory, where the equivalence relation between two categorified group-sets is given by a particular equivalence between the underlying categories.  We also exhibit  some of the differences between the old Burnside theory and the new one.
Lastly, we briefly explain how to extend these new techniques and concepts to the context of groupoids, employing the categories of (right) groupoid-sets, aiming by this to give an alternative approach to the classical Burnside ring of groupoids.
\end{abstract}

\thanks{This article was written while Leonardo Spinosa was a member of the ``National Group for Algebraic and Geometric Structures, and their Applications'' (GNSAGA-INdAM).
Research supported by the Spanish Ministerio de Econom\'{\i}a y Competitividad  and the European Union FEDER, grant MTM2016-77033-P}

\maketitle
\vspace{-0.8cm}
\begin{small}
\tableofcontents
\end{small}

\pagestyle{headings}
\vspace{-1.2cm}

\section*{Introduction}

The Burnside theory for finite groups is a classical subject in the context of group representation theory and it has been introduced for the first time in~\cite{Burnside:1911}.
Afterwards, this topic has been further developed in~\cite{SolBurnAlgFinGr} and~\cite{DressSolvGrps} and it has found many applications in different fields like, for example, homotopy theory (see~\cite{Segal:1970}).
Aiming to extend this theory to the framework of ``finite'' groupoids, we discovered, in~\cite{KaoutitSpinosaBurn}, that the Burnside contravariant functor does not distinguish between  a given groupoid and its bundle of isotropy groups.
Specifically, it has been realized in~\cite{KaoutitSpinosaBurn} that, under appropriate finiteness conditions,  the classical Burnside ring of a given groupoid is isomorphic to the product of the Burnside rings of its isotropy group types, although, not in a canonical way.  
The crux is that the isomorphism relation between finite (and not finite) groupoid-sets leads only to the consideration of (right) cosets by subgroupoids with a single object, and this, somehow, obscures the whole structure of the handled groupoid.
In other words, the classical Burnside ring of a (finite) groupoid dose not codifies, as in the classical case,   information about the whole ``lattice'' of subgroupoids, since the subgroupoids with several objects do not show up at all.

This paper is an attempt to give another approach to the Burnside ring of groupoids, by considering the category (2-category in fact) of internal categories inside the category of (right) groupoid-sets.  The objects of this category (the 0-cells) are referred to as \emph{categorified groupoid-sets} and, by abuse of terminology, the associated ring is called the \emph{categorified Burnside ring} of the given groupoid.
It is noteworthy to mention that, albeit we get in this way a commutative ring that strictly contains the classical one,  we show that this new ring also can be decomposed, in a not canonical way, as a product of rings, which are the  categorified Burnside rings of the isotropy group types of the groupoid.
This makes manifest that also the idea of employing the categorification of the notion of groupoid-sets does not adjust to the groupoid structural characteristics.  Nevertheless, in the case of  groups (i.e., groupoids with  only one object), this categorification approach leads to a new Burnside ring to enter into the picture (see Example~\ref{exam:ExmGrps}).

The concept of categorification has been explained extensively in~\cite[pag.~495]{BaezCransHDimAlgVILiAlg} and~\cite{BaezLaudaHDimAlgV2Grps}. Roughly speaking,  the idea behind is to replace the underlying set of an algebraic structure (resp.~the structure maps), like a group, with a certain category (resp.~with functors), with the goal of obtaining a new  structure which could help to understand the initial one.
In the case of the category of groups, for example, the categorificaton process produces the notion of \(2\)-group (see~\cite{BaezLaudaHDimAlgV2Grps}), which has been proved to be equivalent to the concept of crossed module introduced by Whitehead in~\cite{WhiteheadNoteAddingRel} and~\cite{WhiteheadCombHomII}.
To performe the categorification of a structure, the notions of internal category, internal functor and internal natural transformation, introduced in~\cite{EhresmannCatStruc}, \cite{EhresmannCatStrucIII} and~\cite{EhresmannIntroStrucCat}, are crucial and they will be used extensively in this work.

The main idea of this paper is to categorify the notion of group action on a set to obtain a particular category with a group action on both the set of objects and of morphisms.
Moreover, the source, target, identity and composition maps of this category will have to be compatible with the group action.
Regarding the usual right translation groupoid, it will be replaced by a right translation double category (an internal category in the category of small categories) to illustrate the new higher dimensional situation.

After this, we elaborate a new Burnside theory, based on a particular notion of weak equivalence between these new categories endowed with a group action.
We refer the reader to~\cite{Burnside:1911} and~\cite{SolBurnAlgFinGr} for the classical Burnside theory of groups.
We note that there are, in the literature, other generalizations of the classical Burnside theory: see for instance, \cite{OdaYoshida2001}, \cite{Hartman/Yalcin:2007}, \cite{Diaz/Libman:2009} and \cite{GunnellsEtall}.

In the last section we briefly explain how to extend this idea of categorification to the case of groupoid actions and how it can be reduced to the case of group actions.
We refer the reader to~\cite{KaoutitSpinosaBurn} for the classical Burnside theory applied to groupoids.

\section{Preliminaries and basic definitions}\label{sec:PBN}
In this section we recall the notion of internal categories in small categories with pull-backs\footnote{A category is called small when it is a hom-sets  category \cite{Lane1998}, and  its class of objects is actually a set.\label{Fnote:I}}, and use  this notion to introduce what we will call the category of \emph{right categorified group-sets}.
This is a kind  of a categorificaton of the usual notion of right group-set object (see~\cite{BoucBisetFunctors}). 

Given two  functions \(\alpha \colon A \longrightarrow D\) and \(\beta \colon B \longrightarrow D\), we will use the notation:
\[ 
A \pdue{\times}{\Sscript{\alpha}}{\Sscript{\beta}}  	B = \Set{ \left({a, b}\right) \in A \times B | \alpha\left({a}\right)= \beta\left({b}\right)  }.
\]
This set is well known as the \imp{fibre product} of \(\alpha\) and \(\beta\) and it is the pull-back of the maps \(\alpha\) and \(\beta\) in the category of sets. This means that the following commutative diagram, 
\begin{equation}
\label{Eq:Cartesian}
\begin{gathered}
\xymatrix{  A \pdue{\times}{\Sscript{\alpha}}{\Sscript{\beta}}  B \ar@{->}^-{\pr_2}[rr]  \ar@{->}_-{\pr_1}[d]  & & B  \ar@{->}^-{\beta}[d] \\ A \ar@{->}^-{\alpha}[rr] & & D }
\end{gathered}
\end{equation}
is cartesian, where $\pr_1$ and $\pr_2$ are the canonical projections. This notion can be also adopted in a categorical setting replacing sets with objects in a given category with pull-backs. This, in particular, applies to the category of small categories and functors between them.
The following Definition~\(\ref{dIntCat}\) is taken from~\cite[pag.~495]{BaezCransHDimAlgVILiAlg}.

\begin{definition}\label{dIntCat}
Given a category with pull-back \(\Cat{C}\), we define an \imp{internal category \(\Cat{X}\) in \(\Cat{C}\)} as a couple of objects \(\Xob\) and \(\Xmo\) of \(\Cat{C}\) and morphisms
\[ \xymatrix{\Xmo \ar@<1ex>[rr]^{\mathsf{t}\ped{\Cat{X} } } \ar@<-1ex>[rr]_{\mathsf{s}\ped{\Cat{X} } } && \Xob \ar[rr]^{\iota\ped{\Cat{X} } } && \Xmo &&& \Xmod \ar[lll]_{m\ped{\Cat{D} } }  : = \Xmo \pdue{\times}{\mathsf{s}\ped{\Cat{X} } }{\mathsf{t}\ped{\Cat{X} } } \Xmo },
\]
where \(\mathsf{s}\ped{\Cat{X}}\) and \(\mathsf{t}\ped{\Cat{X}}\) are called the \imp{source} and the \imp{target morphisms}, respectively, \(\iota\ped{\Cat{X}}\) is called the \imp{identity morphism} and \(m\ped{\cX}\) is called the \imp{composition morphism}, or ``\imp{multiplication morphism}'', such that the following diagrams are commutative (note that we will use the notation \(\Xmot  : = \Xmo \pdue{\times}{\mathsf{s}\ped{\Cat{X} } }{\mathsf{t}\ped{\Cat{X} } } \Xmo \pdue{\times}{\mathsf{s}\ped{\Cat{X} } }{\mathsf{t}\ped{\Cat{X} } } \Xmo \)).
\[ \begin{gathered}
\xymatrix{
\Xob \ar[rr]^{\iota\ped{\Cat{X} } } \ar[d]_{\iota\ped{\Cat{X} } } \ar[rrd]^{\id{\Xob}  }  & & \Xmo  \ar[d]^{\mathsf{s}\ped{\Cat{X} } }  \\
\Xmo \ar[rr]^{\mathsf{t}\ped{\Cat{X} } } & & \Xob 
}
\quad 
\xymatrix{
\Xmod \ar[r]^{m\ped{\Cat{X} } } \ar[d]_{\pr\sped{1} } &  \Xmo  \ar[d]^{\mathsf{t}\ped{\Cat{X} } }  \\
\Xmo \ar[r]^{\mathsf{t}\ped{\Cat{X} } } &  \Xob 
}
\quad 
\xymatrix{
\Xmod \ar[r]^{m\ped{\Cat{X} } } \ar[d]_{\pr\sped{2} } &  \Xmo  \ar[d]^{\mathsf{s}\ped{\Cat{X} } }  \\
\Xmo \ar[r]^{\mathsf{s}\ped{\Cat{X} } } &  \Xob 
}
\quad 
\xymatrix{
\Xmot \ar[rr]^{m\ped{\Cat{X} } \times \id{\Xmo}  } \ar[d]|{\id{\Cat{X} } \times m\ped{\Xmo } } & & \Xmod  \ar[d]^{ m\ped{\Cat{X} } }  \\
\Xmod \ar[rr]^{m\ped{\Cat{X} } } & & \Xmo 
} \\
 \xymatrix{
\Xob \pdue{\times}{\id{\Xob} }{\mathsf{t}\ped{\Cat{X} } } \Xmo \ar[rr]^{\iota\ped{\Cat{X} } \times \id{\Xmo}  } \ar[rrd]_{\pr\sped{2} } && \Xmo \pdue{\times}{\mathsf{s}\ped{\Cat{X} } }{\mathsf{t}\ped{\Cat{X} } } \Xmo  \ar[d]^{m\ped{\Cat{X} } } && \Xmo \pdue{\times}{\mathsf{s}\ped{\Cat{X} } }{\id{\Xob}  } \Xob \ar[ll]_{\id{\Xmo} \times \iota\ped{\Cat{X} } } \ar[lld]^{\pr\sped{1}} \\
 && \Xmo  && 
}
\end{gathered}
\]
\end{definition}

Internal categories in $\Cat{C}$ can be viewed as 0-cells in a certain 2-category: 

\begin{definition}
\label{dIntFun}
Let  \(\Cat{C}\) be a  category with pull-back, and consider two internal categories \(\Cat{X}\) and \(\Cat{Y}\) in \(\Cat{C}\).
We define an \imp{internal functor \(F \colon \Cat{X} \longrightarrow \Cat{Y}\) in \(\Cat{C}\)} as a couple of morphism \(F\sped{0} \colon \Xob \longrightarrow \Yob\) and \(F\sped{1} \colon \Xmo \longrightarrow  \Ymo\) such that the following diagrams are commutative, where we use the notation \(F\sped{2} = F\sped{1} \times F\sped{1} \colon \Xmod \longrightarrow \Ymod\).
\[ \xymatrix{
\Xmo \ar[r]^{\mathsf{s}\ped{\Cat{X} } } \ar[d]_{F\sped{1} } &  \Xob  \ar[d]^{F\sped{0} }  \\
\Ymo \ar[r]^{\mathsf{s}\ped{\Cat{Y} } } &  \Yob 
}
\qquad  \qquad
\xymatrix{
\Ymo \ar[r]^{\mathsf{t}\ped{\Cat{X} } } \ar[d]_{F\sped{1} } & \Xob  \ar[d]^{F\sped{0} }  \\
\Ymo \ar[r]^{\mathsf{t}\ped{\Cat{Y} } } & \Yob 
}
\qquad  \qquad
 \xymatrix{
\Xob \ar[r]^{\iota\ped{\Cat{X} } } \ar[d]_{F\sped{0} } &  \Xmo  \ar[d]^{F\sped{1} }  \\
\Yob \ar[r]^{\iota\ped{\Cat{Y} } } &  \Ymo 
}
\qquad \qquad
\xymatrix{
\Xmod \ar[r]^{m\ped{\Cat{X} } } \ar[d]_{F\sped{2} }  & \Xmo  \ar[d]^{F\sped{1} }  \\
\Ymod \ar[r]^{m\ped{\Cat{Y} } } &  \Ymo 
}
\]
\end{definition}

\begin{definition}\label{dIntNatTra}
Given  \(\Cat{C}\) as above, let's consider two internal categories \(\Cat{X}\) and \(\Cat{Y}\) in \(\Cat{C}\), two internal functors \(F, G \colon \cX \longrightarrow \cY\) in \(\Cat{C}\), we define an \imp{internal natural transformation \(\alpha \colon F \longrightarrow G\) in \(\Cat{C}\)} as a morphism \(\alpha \colon \Xob \longrightarrow \Ymo\) in \(\Cat{C}\) such that the following diagrams are commutative, where \(\Delta\) denotes the morphism given by the universal property of the pull-back in \(\Cat{C}\).
\[ \xymatrix@R=30pt{
\Xob \ar[rrd]_{ F\sped{0} }  \ar[rr]^{\alpha} &  & \Ymo \ar[d]^{\mathsf{s}\ped{\mathcal{X} } }  \\
&& \Yob        }
\qquad  \qquad
\xymatrix@R=30pt{
\Xob \ar[rrd]_{ G\sped{0} }  \ar[rr]^{\alpha} &  & \Ymo \ar[d]^{\mathsf{t}\ped{\mathcal{X} } }  \\
&& \Yob        }
\qquad \qquad
 \xymatrix@R=30pt{
\Xmo \ar[rr]^{\Delta\left( G\sped{1}, \, \alpha \mathsf{s}\ped{\Cat{X} } \right) } \ar[d]|{\Delta\left({\alpha \mathsf{t}\ped{\Cat{X} }, \,F\sped{1} }\right)} & & \Ymod  \ar[d]^{m\ped{\Cat{Y} } }  \\
\Ymod \ar[rr]^{m\ped{\Cat{Y} } } & & \Ymo
}
\]
\end{definition}

Internal natural transformations can be composed (horizontally or vertically) in a similar way to ordinary natural transformation and we refer to~\cite[Pag. 498]{BaezCransHDimAlgVILiAlg} for a more detailed explanation.

We note that the experienced reader will not fail to see the similarities between the theory of internal categories and enriched category theory (see \cite{KellyErichCatTh2005}).
In this work, however, we chose to keep the internal categories approach already used  in~\cite{BaezLaudaHDimAlgV2Grps} and~\cite{BaezCransHDimAlgVILiAlg}.

For a given group $G$, we will denote by $\rset{G}$ its category of right $G$-sets.  
Morphisms in this category are referred to as \emph{$G$-equivariant maps}.
The pull-backs in $\rset{G}$ are given as follows: Let us assume that we have a  diagram of $G$-equivariant maps $\alpha: X \rightarrow Z \leftarrow Y: \beta$.
Then the pull-back set $X \pdue{\times}{\Sscript{\alpha}}{\Sscript{\beta}} Y$ is a $G$-set via the action $(x,y) g= (xg,yg)$, with $x \in X$, $y \in Y$ and $g \in G$,  such that the analogue diagram of equation~\eqref{Eq:Cartesian} becomes a cartesian square of right $G$-sets.  

\begin{definition}\label{def:SiSets}
Given a group \(G\), we define a \imp{right categorified \(G\)-set} as an internal category in the category of right \(G\)-sets, a \imp{morphism of right categorified \(G\)-sets} as an internal functor in the category of right \(G\)-sets and a \imp{\(2\)-morphism between morphisms of right categorified \(G\)-sets} as an internal natural transformation in \(\rset{G}\).
In this way, thanks to~\cite{EhresmannCatStruc} and~\cite[Prop.~2.4]{BaezCransHDimAlgVILiAlg}, we obtain a \(2\)-category that, by abuse of notations,  we denote with \(\CRset{G}\). The category of left categorified $G$-set is similarly defined, and clearly isomorphic to the right one. We will also employ the terminology \emph{categorified right group-set}, whenever the handled group is not relevant for the context.  In all what follows \emph{categorified group-sets} stands for right ones.
\end{definition}

\begin{remark}\label{rem:SiSets}
In fact, the category of  categorified $G$-sets is the category that ``contains'' $\CRset{G}$ as full subcategory. Such a category  is defined as the category of functors  $[\bd{\Delta}^{op}, \rset{G}]$ from the opposite category of  $\bd{\Delta}$ of finite sets $\Delta_{n}=\{ 0, 1, \cdots, n\}$, with increasing maps as arrows, to the category $\rset{G}$ of right $G$-sets. This somehow justifies the employed terminology in Definition~\ref{def:SiSets}. 
The choice that we made, in working with the category $\CRset{G}$ instead of $[\bd{\Delta}^{op}, \rset{G}]$,  has its origin, apparently, in some of the  difficulties that the whole category of categorified right $G$-sets presents, especially in developing a certain kind of Burnside theory, as we will see  in the sequel for $\CRset{G}$. Nevertheless, it is noteworthy to mention that many of the results stated below for $\CRset{G}$ might be directly extended  to the whole category of categorified $G$-sets. 
\end{remark}

\begin{remark}\label{lOrdinaryCatTheory}
Following Definition \ref{def:SiSets}, we have to note that categorified \(G\)-sets, morphisms of categorified \(G\)-sets, and the relative \(2\)-morphisms constitute, respectively, categories, functors and natural transformations in the usual sense. This means that many definitions of the usual category theory are valid in this setting. For example, given a  \(G\)-set \(\Cat{X}\), an element \(f \in \Xmo\) is called an \imp{isomorphism} if there is \(h \in \Xmo\) such that \(hf = \iota\ped{\Cat{X} }\left({\mathsf{s}\ped{\Cat{X} }(f)}\right)\) and \(fh = \iota\ped{\Cat{X} }\left({\mathsf{t}\ped{\Cat{X} }(f)}\right)\).
In this way, we obtain a forgetful functor from the \(2\)-category of internal categories in a category \(\Cat{C}\) to the category of ordinary small categories.
\end{remark}

Direct consequences of Definition \ref{def:SiSets}, are as follows. Let  \(\Cat{X}\) be a  categorified \(G\)-set, 
given \(a, b \in \Xob\), we will use the notation
\[ 
\Cat{X}(a,b) = \Big\{ f \in \Xmo |\;  \mathsf{s}\ped{\Cat{X} }(f) = a \quad \text{and} \quad
\mathsf{t}\ped{\Cat{X} }(f) = b \Big\}.
\]

As was mentioned above, the set $\Xmod$ admits in a canonical way an action of $G$-set, given by  $(p, q) .  g\,=\, (pg, qg)$, for every  $(p, q) \in \Xmod$ and $g \in G$.  In this way, the fact that the composition map $m\ped{\Cat{X}}$ is a $G$-equivariant leads to the following equalities:
\begin{equation}\label{Eq:circ}
(p \circ q) \, . \,  g \,\, =\,\, (p g) \circ (qg)
\end{equation}
for every composable arrows $p, q$, and every element $g \in G$. In this direction, we have that a morphism $p \in \Xmo$ is an isomorphism if and only if $pg$ is an isomorphism for some $g \in G$. Furthermore, for any element $a \in \Xob$ and $g \in G$, we have that the map 
\[
\cX(a,a) \longrightarrow \cX(ag,ag), \quad \Big( \ell \longmapsto \ell\, g\Big)
\]
is a morphism of monoids (or semigroups).

In the rest of the paper we will consider the category  \(\CRset{G}\), defined in Definition~\ref{def:SiSets}, mainly as a category (the \(2\)-category level will be used to define the concept of weak equivalence in Definition~\(\ref{dWeakEquiv}\) below).

\begin{remark}
Let be \(\Cat{X} \in \CRset{G}\): we consider the decomposition of \(\Xob\) into orbits (i.e., transitive  \(G\)-sets) \(\Xob = \biguplus\ped{\alpha \in A} \left[{x_\alpha}\right] G\) with \(x_\alpha \in \Xob\) for each \(\alpha \in A\), where $A$ is a certain set of representative elements. 
Since \(\iota\ped{\Cat{X}}\) is a morphism of  \(G\)-sets, for each \(\alpha \in A\) we obtain \(\iota\ped{\Cat{X} }\left({\left[{x_\alpha}\right] G }\right) = \left[ \iota\ped{\Cat{X} }\left({x_\alpha}\right) \right] G\). Therefore, we can state that
\[ \Xmo = \left({\biguplus\ped{\alpha \, \in \,  A} \left[ \iota\ped{\Cat{X} } \left({ x_\alpha }\right) \right] G }\right) \Uplus \left({\biguplus\ped{\beta\,  \in\,  B} \left[ y_\beta \right] G }\right) 
\]
with \(x_\alpha \in \Xob\) for each \(\alpha \in A\) and \(y_\beta \in \Xmo \setminus \iota\ped{\Cat{X} }\left({\Xob}\right)\) for each \(\beta \in B\), with \(B \cap A = \emptyset\).
As a consequence we can state that \(\Cat{X}\) is a discrete category if and only if \(\Xmo = \biguplus\ped{\alpha\, \in\,  A} \left[ \iota\ped{\Cat{X} } \left({ x_\alpha }\right) \right] G\).
\end{remark}

\section{The symmetric monoidal structures of categorified group-sets}\label{secMonStruc}
We describe the two symmetric monoidal structures underlying the category of categorified group-sets: one is given by the disjoint union, i.e., the coproduct   $\Uplus$ , and the other by the product $\times$. A distributivity laws, in the appropriate sense, between these two structures, is what is known in the literature as a Laplaza category: see for instance~\cite[Appendices]{KaoutitSpinosaBurn} and the references therein. This fact will be mentioned nowhere below, however, and it will be implicitly used as long as needed.

Let us fix a  group \(G\); a  categorified $G$-set will be denoted by $\Cat{X}:=(\Cat{X}_{0}, \Cat{X}_1, {\sf{s}}_{\Cat{X}}, {\sf{t}}_{\Cat{X}}, \iota_{\Cat{X}}, m_{\Cat{X}})$.  In this way, it is evident that the object \(\emptyset = ( \emptyset,  \emptyset, \emptyset, \emptyset, \emptyset, \emptyset)\) is initial in \(\CRset{G}\), that the disjoint union \(\Uplus\) is a coproduct in \(\CRset{G}\) and that \(\left({\CRset{G}, \Uplus, \emptyset}\right)\) is a strict monoidal category.
Specifically, given objects two categorified $G$-sets  \(\Cat{X}, \Cat{Y}\), we define
\[ \mathsf{s}\ped{\Cat{X} \uplus \Cat{Y} } = \mathsf{s}\ped{\Cat{X} } \Uplus \mathsf{s}\ped{\Cat{Y} }, 
\quad \mathsf{t}\ped{\Cat{X} \uplus \Cat{Y} } = \mathsf{t}\ped{\Cat{X}} \Uplus \mathsf{t}\ped{\Cat{Y} }, 
\quad  \iota\ped{\Cat{X} \uplus \Cat{Y} } = \iota\ped{\Cat{X}} \Uplus \iota\ped{\Cat{Y} }
\quad \text{and} \quad
m\ped{\Cat{X} \uplus \Cat{Y} } = m\ped{\Cat{X} } \Uplus m\ped{\Cat{Y} }.
\]
Moreover, given morphisms \(\varphi \colon \Cat{X} \longrightarrow \Cat{Y}\) and \(\psi \colon \Cat{A} \longrightarrow \Cat{B}\) in \(\CRset{G}\), we define the morphism \(\varphi \Uplus \psi \colon \Cat{X} \Uplus \Cat{A} \longrightarrow \Cat{Y} \Uplus \Cat{B}\)  as the couple of morphisms \(\left( \varphi \Uplus \psi \right)\sped{0} = \varphi\sped{0} \Uplus \psi\sped{0}\) and \(\left( \varphi \Uplus \psi \right)\sped{1} = \varphi\sped{1} \Uplus \psi\sped{1}\) in \(\Rset{G}\).

Next, we  construct a monoidal structure \(\left({\CRset{G}, \times, \one}\right)\), where \(1\) is the  group with a single element.
Given \(\Cat{X}, \Cat{Y}  \in \CRset{G}\) we can consider the Cartesian product of the underlying categories $\cX$ and $\cY$ in the category of small categories.  Thus, we define \(\left({\Cat{X} \times \Cat{Y} }\right)\sped{0} =  \Xob \times \Yob\), \(\left({\Cat{X} \times \Cat{Y} }\right)\sped{1} =  \Xmo \times \Ymo\)
and
\[ \Cat{X} \times \Cat{Y}  = \left( \Xob \times \Yob, \,  \Xmo \times \Ymo,  \, \mathsf{s}\ped{\Cat{X} \times \Cat{Y}}, \, \mathsf{t}\ped{\Cat{X} \times \Cat{Y}}, \, \iota\ped{\Cat{X} \times \Cat{Y}}, \, m\ped{\Cat{X} \times \Cat{Y}} \right)
\]
where \(\mathsf{t}\ped{\Cat{X} \times \Cat{Y}} = \left({\mathsf{t}\ped{\Cat{X} }, \mathsf{t}\ped{ \Cat{Y}} }\right)\), \(\mathsf{s}\ped{\Cat{X} \times \Cat{Y}} = \left({\mathsf{s}\ped{\Cat{X} }, \mathsf{s}\ped{ \Cat{Y}} }\right)\) and \(\iota\ped{\Cat{X} \times \Cat{Y}} = \left({\iota\ped{\Cat{X} }, \iota\ped{ \Cat{Y}} }\right)\). 
Regarding the composition, we define
\[ \begin{aligned}{m\ped{\Cat{X} \times \Cat{Y}}}  \colon & {\left({\Cat{X} \times \Cat{Y} }\right)\sped{2} } \longrightarrow {\left({\Cat{X} \times \Cat{Y} }\right)\sped{1} } \\ & {\big( (x,y), (a,b) \big)}  \longrightarrow {\left({m\ped{\Cat{X}}(x,a), m\ped{\Cat{Y}}(y,b)}\right). }\end{aligned} 
\]
It is immediate to verify that \(\mathsf{s}\ped{\Cat{X} \times \Cat{Y}}\), \(\mathsf{t}\ped{\Cat{X} \times \Cat{Y}}\), \(\iota\ped{\Cat{X} \times \Cat{Y}} = \left({\iota\ped{\Cat{X} }, \iota\ped{ \Cat{Y}} }\right)\) and \(m\ped{\Cat{X} \times \Cat{Y} }\) are morphisms in \(\rset{G}\).
We have to prove that the diagrams of Definition~\(\ref{dIntCat}\) about \(\Cat{X} \times \Cat{Y}\) are commutative, but this is a direct verification and follows from the analogous diagrams about \(\Cat{X}\) and \(\Cat{Y}\).

Now, given morphisms \(\varphi \colon \Cat{X}   \longrightarrow \Cat{Y}\) and \(\psi \colon \Cat{A}  \longrightarrow \Cat{B}\) in \(\CRset{G}\) we define the morphism
\[ \varphi\times  \psi \colon \Cat{X}  \times\Cat{A}   \longrightarrow \Cat{Y}  \times \Cat{B}
\]
in \(\CRset{G}\) as the couple of morphisms \(\left( \varphi \times \psi \right)\sped{0} = \varphi\sped{0} \times \psi\sped{0} \) and \(\left( \varphi \times  \psi \right)\sped{1} = \varphi\sped{1} \times  \psi\sped{1}\) in \(\Rset{G}\).
We have to prove that the diagrams of Definition~\(\ref{dIntFun}\) about \(\varphi \times \psi\) are commutative, but this is a direct verification and follows from the analogous diagrams about \(\varphi\) and \(\psi\).
It is now obvious that we have constructed a functor \(\left( - \times  - \right)  \colon \CRset{G} \times \CRset{G} \longrightarrow \CRset{G}\).

In order to complete the monoidal structure on $\CRset{G}$, we need to construct natural isomorphisms
\[ 
\Phi \colon \left({ \id{\CRset{G} }  \times \one }\right) \longrightarrow \id{\CRset{G} } ,
\qquad 
\Psi \colon \left({ \one   \times  \id{\CRset{G} }   \longrightarrow \id{\CRset{G} } }\right)
\]
and the associator 
\[ \left({ \left({ - \times - }\right) \times - }\right) \longrightarrow \left({ - \times \left({- \times -}\right) }\right) .
\]
The associator is the identity which is clearly a natural isomorphism and satisfies the pentagonal constraint.
We will construct only \(\Phi\) because \(\Psi\) can be realized in a similar way.
Let \(\cX\) be a categorified $G$-set: for \(i=0,1\) we have  the isomorphisms of  \(G\)-sets
\[ \begin{aligned}{\Phi\left({\Cat{X} }\right)\sped{i} }  \colon & {\Xmoi \times \one } \longrightarrow {\Xmoi} \\ & {(a,1)}  \longrightarrow {a.}\end{aligned} 
\]
It is a direct verification to check that \(\Phi\left({\Cat{X}}\right)\) is a morphism of categorified \(G\)-sets.
Now we just have to prove that
\[  \Phi \colon \left({ \id{\CRset{G} }  \times   \one }\right) \longrightarrow \id{\CRset{G} } 
\]
is a natural transformation: let \(\alpha \colon \Cat{X} \longrightarrow \Cat{Y}\) be a morphism in \(\CRset{G}\).
We have to show that the diagram
\[ \xymatrix{
\Cat{X} \times \one  \ar[rr]^{\Phi\left({\Cat{X} }\right) } \ar[d]_{\alpha \times \one } & & \Cat{X}   \ar[d]^{\alpha}  \\
\Cat{Y} \times \one  \ar[rr]^{\Phi\left({\Cat{Y} }\right) } & & \Cat{Y} 
}
\]
is commutative which is equivalent to say that the diagram
\[ \xymatrix{
\Xmoi \times \one  \ar[rr]^{\Phi\left({\Cat{X} }\right)\sped{i}  } \ar[d]_{\alpha\sped{i}  \times \one } & & \Xmoi   \ar[d]^{\alpha\sped{i} }  \\
\Ymoi \times \one  \ar[rr]^{\Phi\left({\Cat{Y} }\right)\sped{i}  } & & \Ymoi 
}
\]
is commutative for \(i=0,1\), but this is immediate to check.
Lastly, it is obvious that \(\Phi\) and \(\Psi\) satisfies the triangular identities, and this completes the claimed constructions.

In the same direction, there is an ``inclusion'' functor from the category of usual right \(G\)-sets to the category of  categorified \(G\)-sets. Specifically, we have a fully faithful functor
\begin{equation}
\label{eIncFun}
\begin{aligned}{\SInc{G} }  \colon & {\rset{G} } \longrightarrow {\CRset{G} } \\ & {X }  \longrightarrow {\Big( X, X, \mathsf{s}\ped{X }, \mathsf{t}\ped{X }, \iota\ped{X }, m\ped{X } \Big) }\end{aligned} 
\end{equation} 
where \(\mathsf{s}\ped{X} = \mathsf{t}\ped{X} = \iota\ped{X} = \id{X}\) and \(m\ped{X}= \pr\sped{1} \colon X\sped{2} \longrightarrow X\), that is, the underlying category of the image of $X$ is the discrete one.
The behaviour of \(\SInc{G}\) on morphisms is obvious.
Basically the idea is that the image of \(\SInc{G}\) is given by discrete categories.
Moreover, we will use the abuse of notation \(\one = \SInc{G}\left({\one}\right)\).

Summing up, we have,  therefore, proved the following result.

\begin{proposition}
Given a group \(G\), \(\left({\CRset{G}, \times, \one}\right)\) is a symmetric monoidal category such that the functor $\cI\text{-}\cG$ becomes a strict symmetric monoidal functor.
\end{proposition}

Lastly, let us discuss  the distributivity between the two operations $\uplus$ and $\times$ in the category $\CRset{G}$. We know that we have two monoidal structures \(\left({\CRset{G}, \Uplus, \emptyset}\right)\) and \(\left({\CRset{G}, \times, \one}\right)\). As it was mentioned above, it is necessary to prove the distributivity of \(\times\) over \(\Uplus\).
Let be \(\Cat{X}, \Cat{Y}, \Cat{A} \in \CRset{G}\): we have to construct a morphism
\[ \lambda \colon \left[{ \Cat{X} \Uplus \Cat{Y}  }\right] \times \Cat{A}  \longrightarrow \left[{ \Cat{X}    \times \Cat{A} }\right] \Uplus \left[{ \Cat{Y}  \times \Cat{A} }\right]
\]
in \(\CRset{G}\).
We define it as the couple of morphisms in \(\Rset{G}\), for \(i=0,1\),
\[ \lambda\sped{i} \colon \left[{\Cat{X}\sped{i} \Uplus \Cat{Y}\sped{i}  }\right] \times \Cat{A}\sped{i}  \longrightarrow \left[{\Cat{X}\sped{i}    \times \Cat{A}\sped{i}  }\right] \Uplus \left[{ \Cat{Y}\sped{i}   \times \Cat{A}\sped{i}  }\right]
\]
that send \((a,b)\) to \((a,b)\) both if \(a \in \Xmoi\) and if \(a \in \Ymoi\).
The proof of the commutativity of the diagrams of Definition~\(\ref{dIntFun}\) is now obvious.

\section[Weak equivalences]{Weak equivalences and 2-morphisms}\label{secCAtGSetWEquiv}

We introduce the notion of \emph{weak equivalence}   in the category of categorified group-sets. Then, we show that  both operations $\Uplus$ and $\times$ are compatible with the weak equivalence relation, as well as  with 2-\emph{morphisms}. This will be crucial to build up the categorified Burnside functor in the forthcoming sections.

Fix a group $G$ and consider its category $\CRset{G}$ of categorified $G$-sets. Given morphisms in \(\CRset{G}\)
\[ \varphi, \psi \colon \Cat{X}  \longrightarrow \Cat{Y}
\qquad \text{and} \qquad
\varepsilon,  \eta \colon \Cat{A}  \longrightarrow \Cat{B},\]
let's consider \(2\)-morphisms in \(\CRset{G}\) \(\alpha \colon \varphi \longrightarrow \psi\) and \(\beta \colon \varepsilon \longrightarrow \eta\): we have
\[  \varphi \Uplus \varepsilon , \,  \psi \Uplus \eta \colon \Cat{X} \Uplus \Cat{A} \longrightarrow \Cat{Y}  \Uplus \Cat{B}.
\]
We define the \(2\)-morphism \(\alpha \Uplus \beta\) in \(\CRset{G}\) as the following morphism in \(\CRset{G}\):
\[ \begin{array}{cccc}
\alpha \Uplus \beta  \colon & \Xob \Uplus \Aob  & \longrightarrow & \Ymo \Uplus \Bmo \\
&  \Xob \ni x & \longrightarrow & \alpha(x) \in \Ymo \\
&  \Aob \ni x & \longrightarrow & \beta(x) \in \Bmo .
\end{array}
\]
The verification that \(\alpha \Uplus \beta\) renders the diagrams of Definition~\(\ref{dIntNatTra}\) commutative is immediate and derives from the relative diagrams regarding \(\alpha\) and \(\beta\).
Now, we consider \(\psi' \colon \Cat{X}  \longrightarrow \Cat{Y}\) and \(\eta' \colon \Cat{A}  \longrightarrow \Cat{B}\), morphisms in \(\CRset{G}\), and \(\alpha' \colon \psi \longrightarrow \psi'\) and \(\beta'\colon \eta \longrightarrow \eta'\), \(2\)-morphism in \(\CRset{G}\).

\begin{lemma}\label{lema:uplus}
We have
\[ \left({\alpha' \Uplus \beta'}\right) \left({\alpha \Uplus \beta}\right) = \left({\alpha'\alpha}\right) \Uplus \left({\beta' \beta}\right) \colon \varphi \Uplus \varepsilon \longrightarrow \psi \Uplus \eta
\qquad \text{and} \qquad
\id{\varphi \uplus \psi} = \id{\varphi} \Uplus \id{\psi}.
\]
\end{lemma}
\begin{proof}
It is immediate.
\end{proof}

On the other hand, taking $\alpha, \beta$ as above, we have 
\[ 
\varphi \times \varepsilon, \,  \psi \times \eta \colon \Xset{\varsigma} \times \Aset{\lambda}  \longrightarrow \Yset{\theta} \times \Bset{\mu},
\]
and we want to define a \(2\)-morphism
\[ \alpha \times \beta \colon \varphi \times \varepsilon \longrightarrow \psi \times \eta
\]
in \(\CRset{G}\) as the morphism of  \(G\)-sets
\[ \begin{aligned}{\alpha \times \beta }  \colon & {\Xob \times \Aob} \longrightarrow {\Ymo \times \Bmo} \\ & {(x,a)}  \longrightarrow {\left({\alpha(x), \beta(a)}\right) .}\end{aligned} 
\]
In this way, it is sufficient to check that  the  following three diagrams of Definition~\(\ref{dIntNatTra}\)
\[ \xymatrix{
\Xob \times \Aob   \ar[rr]^{\alpha \times \beta} \ar[rrd]_{\varphi\sped{0} \times \varepsilon\sped{0} }  &  & \Ymo \times \Bmo \ar[d]^{\mathsf{s}\ped{\Cat{Y} \times \Cat{B} } } \\
&& \Yob \times \Bob ,    }
\qquad  \qquad
\xymatrix{
\Xob \times \Aob   \ar[rr]^{\alpha \times \beta} \ar[rrd]_{\psi\sped{0} \times \eta\sped{0} }  &  & \Ymo \times \Bmo \ar[d]^{\mathsf{t}\ped{\Cat{Y} \times \Cat{B} } } \\
&& \Yob \times \Bob     }
\]
\[ 
\xymatrix@R=30pt{
\Xmo \times \Amo \ar[rrrrr]^{\Delta\left({\psi\sped{1} \times \eta\sped{1}, \left({\alpha \times \beta}\right) \mathsf{s}\ped{\Cat{X} \times \Cat{A} } }\right) } \ar[d]|{\Delta\left({\left({\alpha \times \beta}\right) \mathsf{t}\ped{\Cat{X} \uplus \Cat{A} }   , \varphi\sped{1} \times \varepsilon\sped{1} }\right) } &&&&& \left({\Ymo \times \Bmo}\right)\sped{2}   \ar[d]|{m\ped{\Cat{Y} \times \Cat{B} } }  \\ 
\left({\Ymo \times \Bmo}\right)\sped{2}  \ar[rrrrr]^{m\ped{\Cat{Y} \times \Cat{B} } } &&&&& \Ymo \times \Bmo .
}
\]
are  commutative. As for the commutativity of the two triangular diagrams is  obvious because \(\mathsf{s}\ped{\Cat{Y} \times \Cat{B} } = \mathsf{s}\ped{\Cat{Y} } \times \mathsf{s}\ped{ \Cat{B} }\) and \(\mathsf{t}\ped{\Cat{Y} \times \Cat{B} } = \mathsf{t}\ped{\Cat{Y} } \times \mathsf{t}\ped{ \Cat{B} }\).
Regarding the commutativity of the third one, we calculate, for each \((x,a) \in \Xmo \times \Amo\),
\[ \begin{aligned}
& \quad m\ped{\Cat{Y} \times \Cat{B} } \Delta\left({\psi\sped{1} \times \eta\sped{1}, \left({\alpha \times \beta}\right) \mathsf{s}\ped{\Cat{X} \times \Cat{A} } }\right) (x,a)  
= m\ped{\Cat{Y} \times \Cat{B} } \Big(  \left({\psi\sped{1}(x), \eta\sped{1}(a)}\right), \left({\alpha \mathsf{s}\ped{\Cat{X} }(x), \beta \mathsf{s}\ped{\Cat{A} }(a)}\right) \Big)  \\
&= \Big( m\ped{\Cat{Y} }\left({\psi\sped{1} (x), \alpha \mathsf{s}\ped{\Cat{X} }(x)}\right), m\ped{\Cat{B} }\left({\eta\sped{1}(a), \beta \mathsf{s}\ped{\Cat{A} }(a)}\right)  \Big) 
= \Big( m\ped{\Cat{Y} }\left({\alpha \mathsf{t}\ped{\Cat{X} }(x), \varphi\sped{1}(x) }\right), m\ped{\Cat{B} }\left({\beta \mathsf{s}\ped{\Cat{A} }(a), \varepsilon\sped{1}(a) }\right)  \Big) \\
&= m\ped{\Cat{Y} \times \Cat{B} }  \Big( \left({\alpha \mathsf{t}\ped{\Cat{X} }(x), \beta \mathsf{t}\ped{\Cat{A} }(a)}\right), \left({\varphi\sped{1}(a), \varepsilon\sped{1}(a)}\right) \Big) 
= m\ped{\Cat{Y} \times \Cat{B} } \Delta\left({\left({\alpha \times \beta}\right) \mathsf{t}\ped{\Cat{X} \times \Cat{A} }   , \varphi\sped{1} \times \varepsilon\sped{1} }\right) (x,a)
\end{aligned} 
\]
and this finishes the proof that $\alpha \times \beta$ is a 2-morphism.

\begin{lemma}\label{lema:Times}
Let $\alpha, \alpha': \varphi \to \psi$ and $\beta, \beta': \varepsilon \to \eta$  as above. Then, we have
\[ 
\left({\alpha' \times \beta'}\right) \left({\alpha \times \beta}\right) = \left({\alpha'\alpha}\right) \times \left({\beta' \beta}\right) \colon \varphi \times \varepsilon \longrightarrow \psi \times \eta
\qquad \text{and} \qquad
\id{\varphi \times \varepsilon} = \id{\varphi} \times \id{\varepsilon}.
\]
\end{lemma}
\begin{proof}
Take an object  \((x,a) \in (\cX \times \cA)_{\Sscript{0}}= \Xob \times \Aob\), then  we have
\[ \begin{aligned}
(\id{\varphi} \times \id{\varepsilon}) (x,a) &= \id{\varphi}(x) \times \id{\varepsilon}(a)
= \iota\sped{\varphi\sped{0}(x)} \times \iota\sped{\varepsilon\sped{0} (a)} 
= \iota\sped{\varphi\sped{0}(x) \times \varepsilon\sped{0}(a)} 
= \iota\sped{(\varphi\sped{0} \times \varepsilon\sped{0}) (x,a) }
= \id{\varphi \times \varepsilon}(x,a).
\end{aligned}
\]
This gives the second stated equality. As for the first one,  we have 
\[ \begin{aligned}
& \quad \left({\alpha' \times \beta'}\right) \left({\alpha \times \beta}\right) (x,a) = m\ped{\Cat{Y} \times \Cat{B} }  \Delta\left({ \alpha' \times \beta', \alpha \times \beta }\right)(x,a) \\
&= m\ped{\Cat{Y} \times \Cat{B} } \left({ \left({\alpha' \times \beta' }\right)(x,a), \left({\alpha \times \beta }\right)(x,a) }\right) 
= m\ped{\Cat{Y} \times \Cat{B} }  \Big( \left({\alpha'(x), \beta'(a)}\right), \left({\alpha(x), \beta(a)}\right) \Big)  \\
&= \Big( m\ped{\Cat{Y} }\left({\alpha'(x), \alpha(x)}\right), m\ped{\Cat{B} }\left({\beta'(a), \beta(a)}\right) \Big) 
= \Big( \left({\alpha' \alpha}\right)(x), \left({\beta' \beta}\right)(a) \Big)
= \left({\left({\alpha' \alpha}\right) \times \left({\beta' \beta}\right) }\right)(x,a).
\end{aligned}
\]
Thus, $\left({\alpha' \times \beta'}\right) \left({\alpha \times \beta}\right) = \left({\alpha'\alpha}\right) \times \left({\beta' \beta}\right)$ and this finishes the proof. 
\end{proof}

Next, we give the main definition of this section.

\begin{definition}\label{dWeakEquiv}
Let be \(\Cat{X}, \Cat{Y}  \in \CRset{G}\).
We say that \(\Cat{X}\) and \(\Cat{Y}\) are \imp{weakly equivalent} and we write \(\Cat{X}  \wea \Cat{Y}\) if there are,   in the category \(\CRset{G}\),  morphisms  \(\varphi \colon \Cat{X} \longrightarrow \Cat{Y}\), \(\psi \colon \Cat{Y} \longrightarrow \Cat{X}\) and \(2\)-isomorphisms \(\alpha \colon \psi \varphi  \longrightarrow \id{\Cat{X}}\), \(\beta \colon \varphi \psi \longrightarrow \id{\Cat{Y}}\), in the sense of Definition  \ref{def:SiSets}.
\end{definition}

The following lemma states that the disjoint union and the fibre product are compatible with the weak equivalence relation. A compatibility criterion that will be used in the sequel. 

\begin{proposition}\label{prop:PusTime}
Let \(\Cat{X}, \Cat{Y}, \Cat{A}\) and \(\Cat{B}\)  be objects in the category \(\CRset{G}\) such that \(\Cat{X} \wea \Cat{Y}\) and \(\Cat{A} \wea \Cat{B}\).
Then, we have the following weak equivalences relations:
\[ \Big[ \Cat{X} \Uplus \Cat{A} \Big] \wea \Big[ \Cat{Y} \Uplus \Cat{B}  \Big]
\qquad \text{and} \qquad
 \Big[ \Cat{X} \times  \Cat{A} \Big] \wea \Big[ \Cat{Y} \times  \Cat{B}  \Big].
\] 
\end{proposition}
\begin{proof}
Attached  to the stated weak equivalences relations, there are morphisms in \(\CRset{G}\) 
\[ 
\varphi \colon \Cat{X} \longrightarrow \Cat{Y}, 
\qquad
 \psi \colon \Cat{Y} \longrightarrow \Cat{X},
\qquad
\eta \colon \cA \longrightarrow \Cat{B} 
\qquad \text{and} \qquad
 \varepsilon \colon \cB \longrightarrow \Cat{A}
\]
such that there are \(2\)-isomorphisms in \(\CRset{G}\)
\[  \alpha \colon \psi \varphi \longrightarrow \id{\Cat{X} } 
\qquad
 \beta \colon \varphi \psi \longrightarrow \id{\Cat{Y} } 
\qquad
 \gamma \colon \varepsilon \eta \longrightarrow \id{\Cat{A} } 
\qquad \text{and} \qquad
 \delta \colon \eta \varepsilon \longrightarrow \id{ \Cat{B} }. 
\]
Applying Lemma \ref{lema:uplus}, we get
\[  \left({\alpha^{-1} \Uplus \gamma^{-1}}\right) \left({\alpha \Uplus \gamma}\right) = \left({\alpha^{-1} \alpha}\right) \Uplus \left({\gamma^{-1} \gamma}\right) 
= \id{\psi \varphi} \Uplus \id{\varepsilon \eta} 
= \id{\psi \varphi \uplus \varepsilon \eta} 
\]
and
\[ \begin{aligned}
& \quad \left({\alpha \Uplus \gamma}\right) \left({\alpha^{-1} \Uplus \gamma^{-1}}\right) = \left({\alpha\alpha^{-1}}\right) \Uplus \left({\gamma \gamma^{-1} }\right) 
= \id{\id{\Cat{X}  } \uplus \id{\Cat{Y}  } } 
= \id{\id{\Cat{X} \uplus \Cat{Y} } }. 
\end{aligned}
\]
Thus
\[ 
\alpha \Uplus \gamma \colon \psi \varphi \Uplus \varepsilon \eta \longrightarrow \id{\Cat{X} \uplus \Cat{A} } 
\]
is a \(2\)-isomorphism in \(\CRset{G}\).
We can prove in the same way that \( \beta \Uplus \delta \colon \varphi \psi \Uplus \eta \varepsilon \longrightarrow \id{\Cat{Y} \uplus \Cat{B} }\) is a \(2\)-isomorphism.  As a consequence, we obtain
\[ \Big[ \Cat{X} \Uplus \Cat{A} \Big] \wea \Big[ \Cat{Y} \Uplus \Cat{B}  \Big].
\]
Lastly, an analogue computation, using this time Lemma~\ref{lema:Times}, shows the relation $\Big[ \Cat{X} \times  \Cat{A} \Big] \wea \Big[ \Cat{Y} \times  \Cat{B}  \Big]$, and this completes the proof. 
\end{proof}

As already noted in Remark~\(\ref{lOrdinaryCatTheory}\), many concept of ordinary category theory can be extended to  categorified  \(G\)-sets.
With reference to~\cite[Pag.~51]{AbstrConcCatJoyCats}, we will now briefly explain how to extend the concept of skeleton of a category.
The proof are essentially the same thus we will omit them.

\begin{definition}
Let be \(\Cat{X}\) and \(\Cat{Y}\) categorified \(G\)-sets.
\begin{enumerate}
\item We say that \(\Cat{Y}\) is a \imp{categorified \(G\)-subset} of \(\Cat{X}\) if the following two conditions are satisfied:
\begin{enumerate}
\item \(\Yob\) and \(\Ymo\) are \(G\)-subset respectively of \(\Xob\) and \(\Xmo\);
\item the structure on \(\Cat{Y}\) is appropriately induced by that on \(\Cat{X}\) by restriction, in the usual sense for subcategories.
\end{enumerate}
\item We say that \(\Cat{Y}\) is a \imp{full  categorified \(G\)-subset} of \(\Cat{X}\) if it is a  categorified \(G\)-subset of \(\Cat{X}\) such that, for each \(a, b \in \Yob\), we have \(\Cat{Y}(a,b) = \Cat{X}(a,b)\).

\item We say that \(\Cat{Y}\) is an \imp{isomorphism-dense  categorified \(G\)-subset} of \(\Cat{X}\) if it is a categorified \(G\)-subset of \(\Cat{X}\) such that, for each \(a \in \Yob\), there is \(b \in \Xob\) such that there is an isomorphism \(f \in \Xmo\), with \(\mathsf{s}\ped{\Cat{X} }(f)= a\) and \(\mathsf{t}\ped{\Cat{X} }(f)= a\).
\end{enumerate}
\end{definition}

\begin{definition}\label{def:Skeleton}
The \imp{skeleteon of a  categorified \(G\)-set} is a full, isomorphism-dense categorified \(G\)-subset in which no two distinct objects are isomorphic.
\end{definition}

Direct consequences of this definition are the following facts:

\begin{proposition}
\label{cWeakEqIsom}
Let $G$ be a group. Then:
\begin{enumerate}
\item Every categorified \(G\)-set has a skeleton.
\item Two skeletons of a  categorified \(G\)-set \(\Cat{X}\) are isomorphic.
\item Every skeleton of a categorified \(G\)-set \(\Cat{X}\) is weakly equivalent to \(\Cat{X}\).
\end{enumerate}
In particular, two  categorified \(G\)-sets \(\Cat{X}\) and \(\Cat{Y}\) are weakly equivalent if and only their skeletons are isomorphic as categorified \(G\)-sets.
\end{proposition}
\begin{proof}
It is immediate.
\end{proof}

\begin{remark}
Thanks to Proposition~\(\ref{cWeakEqIsom}\), given a categorified \(G\)-set \(\Cat{X}\), we can denote one of its skeletons by \(\Ske\left(\Cat{X}\right)\), with the specification that \(\Ske\left(\Cat{X}\right)\) is unique up to isomorphism.
If there is a family \(\left(\Cat{X}_\alpha\right)_{\alpha \in A}\) of categorified \(G\)-subsets of \(\Cat{X}\) such that \(\Cat{X} = \biguplus_{\alpha \in A} \Cat{X}_\alpha\), then we clearly have the following isomorphism of categorified \(G\)-set:
\[ \Ske\left(\Cat{X}\right) = \Ske\left( \biguplus_{\alpha \in A} X_\alpha \right) \cong \biguplus_{\alpha \in A} \Ske\left(\Cat{X}_\alpha\right).
\]
\end{remark}

\begin{definition}
Given a categorified \(G\)-set \(\Cat{X}\), we denote with \(\dCat\left(\Cat{X}\right)\) the disjoint union of the connected components of \(\Cat{X}\) whose skeleton is discrete, that is, has only identities as arrows.
On the other hand, we denote with \(\ndCat\left(\Cat{X}\right)\) the disjoint union of the connected components of \(\Cat{X}\) whose skeleton is not discrete, that is, has at least an arrow that is not an identity.
As a consequence we have \(\Cat{X} = \dCat\left(\Cat{X}\right)  \Uplus \ndCat\left(\Cat{X}\right)\).
\end{definition}


\section{Examples of weak equivalences}\label{secExCatGSets}
In this section we deal with a class of examples of categorified group-sets, and we give a certain criterion of the weak equivalence relation between the objects of this class.  All group-sets are right ones.

Let \(\Cat{C}\) be a small category, \(G\) a group, and \(X\) a  \(G\)-set.
We set \(\Xob = \Cat{C}\sped{0} \times X\) and \(\Xmo = \Cat{C}\sped{1} \times X\).
The $G$-action on the $\Xmoi$'s is defined as \((a,x)g :=(a,xg)\) for each \((a,x) \in \Xmoi\) and \(g \in G\), for \(i=0,1\).
We set
\[ \begin{aligned}{\mathsf{s}\ped{\Cat{X} } }  \colon & {\Xmo} \longrightarrow {\Xob} \\ & {(a,x)}  \longrightarrow {\left({\mathsf{s}\ped{\Cat{C} }(a), g}\right), }\end{aligned} 
\qquad
\begin{aligned}{\mathsf{t}\ped{\Cat{X} } }  \colon & {\Xmo} \longrightarrow {\Xob} \\ & {(a,x)}  \longrightarrow {\left({\mathsf{t}\ped{\Cat{C} }(a), g}\right), }\end{aligned} 
\qquad
\begin{aligned}{\iota\ped{\Cat{X} } }  \colon & {\Xmo} \longrightarrow {\Xob} \\ & {(a,x)}  \longrightarrow {\left({\iota\ped{\Cat{C} }(a), g}\right) }\end{aligned} 
\]
and
\[ \begin{aligned}{m\ped{\Cat{X} } }  \colon & {\Xmod} \longrightarrow {\Xmo} \\ & {\Big((a,x), (b,y) \Big)}  \longrightarrow { \left({m\ped{\Cat{C} }(a,b), x}\right) }\end{aligned} 
\]
with \(\left({\mathsf{s}\ped{\Cat{C}}(a), x}\right) = \mathsf{s}\ped{\Cat{X} }(a,x) = \mathsf{t}\ped{\Cat{X} }(b,y) = \left({\mathsf{t}\ped{\Cat{C} }(b), y}\right)\).
It is evident that \(\mathsf{s}\ped{\Cat{X} }\), \(\mathsf{t}\ped{\Cat{X} }\), \(\iota\ped{\Cat{X}}\) and \(m\ped{\Cat{X} }\) are morphisms of \(G\)-sets thus we just have to prove that the diagrams of Definition~\(\ref{dIntCat}\) about \(\Cat{X}\) are commutative, but this follows immediately by the analogous diagrams about \(\Cat{C}\) (Recall that an ordinary small category, as defined in the footnote of \Cpageref{Fnote:I},  can be considered as an internal category in the category of sets).

Now let \(\Cat{C}\) and \(\Cat{D}\) be small categories, let \(X\) and \(Y\) be  \(G\)-sets and set \(\Xob = \Cat{C}\sped{0} \times X\), \(\Xmo 	= \Cat{C}\sped{1} \times X\), \(\Yob = \Cat{D}\sped{0} \times Y\) and \(\Ymo = \Cat{D}\sped{1} \times Y\).
Let be \(F \colon \Cat{C} \longrightarrow \Cat{D}\) a functor and \(\varphi \colon X \longrightarrow Y\) a morphism in \(\rset{G}\).
We want to define a morphism \((F, \varphi)\) in \(\CRset{G}\) setting \((F, \varphi)\sped{0}= \left({F\sped{0}, \varphi}\right)\) and \((F, \varphi)\sped{1}= \left({F\sped{1}, \varphi}\right)\).
It is evident that \((F, \varphi)\sped{0}\) and \((F, \varphi)\sped{1}\) are morphism of  \(G\)-sets, thus we just have to prove that the diagrams of Definition~\(\ref{dIntFun}\) about \((F, \varphi)\) are commutative, but this follows immediately by the analogous diagrams about \(F\) (an ordinary functor can be considered as an internal functor in the category of sets).

Given another functor \(P \colon \Cat{C} \longrightarrow \Cat{D}\), we consider a natural transformation \(\mu \colon F \longrightarrow P\).
With the notations already introduced, we have a morphism \((P, \varphi) \colon \Cat{X} \longrightarrow \Cat{Y}\) in \(\CRset{G}\) such that \((P, \varphi)\sped{0} = \left({P\sped{0}, \varphi}\right)\) and \((P, \varphi)\sped{1} = \left({P\sped{1}, \varphi}\right)\).
We want to define a \(2\)-morphism \( \left({\mu, \varphi}\right) \colon (F, \varphi) \longrightarrow (P, \varphi)\) given by a morphism of right \(\Cat{G}\)-sets
\[ \begin{aligned}{\left({\mu, \varphi}\right) }  \colon & {\Xob} \longrightarrow {\Ymo} \\ & {(a,x)}  \longrightarrow {(\mu(a), \varphi(x))}\end{aligned} 
\]
therefore, we have to check that the following diagrams commute,
\[ \xymatrix{
\Xob   \ar[rr]^{\left({\mu, \varphi}\right)} \ar[rrd]_{(F, \varphi)\sped{0} }  &  & \Ymo \ar[d]^{\mathsf{s}\ped{\Cat{Y} }} \\
&& \Yob     }
\qquad
\xymatrix{
\Xob   \ar[rr]^{\left({\mu, \varphi}\right)} \ar[rrd]_{(P, \varphi)\sped{0} }  &  & \Ymo \ar[d]^{\mathsf{t}\ped{\Cat{Y} }} \\
&& \Yob     }
\qquad
 \xymatrix{
\Xmo \ar[rrr]^{\Delta \left( (P, \varphi)\sped{1}, \left({\mu, \varphi}\right) \mathsf{s}\ped{\Cat{X} }  \right) } \ar[d]|{\Delta\left({ \left({\mu, \varphi}\right) \mathsf{t}\ped{\Cat{X} }, (F, \varphi)\sped{1} }\right) } &&& \Ymod  \ar[d]^{m\ped{\Cat{Y} } }  \\
\Ymod \ar[rrr]^{m\ped{\Cat{Y} } } &&& \Ymo 
}
\]
but this is just a direct verification.

\begin{remark}
Let be \(\Cat{C}\), \(\Cat{D}\) and \(\Cat{E}\) small categories, \(F \colon \Cat{C} \longrightarrow \Cat{D}\) and \(P \colon \Cat{D} \longrightarrow \Cat{E}\) functors, \(\varphi \colon X \longrightarrow Y\) and \(\psi \colon Y \longrightarrow Z\) morphisms of right \(G\)-sets.
Using the above notation,  we set \(\Xmoi = \Cat{C}\sped{i} \times X\), \(\Ymoi = \Cat{D}\sped{i} \times Y\) and \(\Cat{Z}\sped{i} = \Cat{E}\sped{i} \times Z\) for \(i=0,1\).
It is obvious that \(\id{\Cat{X} } = \left({\id{\Cat{C} }, \id{X} }\right) \colon \Cat{X}  \longrightarrow \Cat{X}\) and \((PF, \psi \varphi) = (P, \psi) (F, \varphi) \colon \Cat{X}  \longrightarrow \Cat{Z}\), are morphisms of categorified $G$-sets.
\end{remark}

\begin{proposition}
\label{rCatNatTransf}
Let be \(\Cat{C}\) and \(\Cat{D}\) small categories, \(X, Y \in \rset{G}\), \(F, P, Q \colon \Cat{C} \longrightarrow \Cat{D}\) functors, \(\varphi \colon X \longrightarrow Y\) a morphism of right \(\Cat{G}\)-sets, \(\mu \colon F \longrightarrow P\) and \(\lambda \colon P \longrightarrow Q\) natural transformations.
We define \(2\)-morphisms
\[ \left({\mu, \varphi}\right)  \colon (F, \varphi) \longrightarrow (P, \varphi),
\qquad
\left({\lambda, \varphi}\right) \colon (P, \varphi) \longrightarrow (Q, \varphi)
\qquad \text{and} \qquad
\left({\lambda \mu, \varphi}\right)  \colon (F, \varphi) \longrightarrow (Q, \varphi)
\]
in \(\CRset{G}\) given by morphisms of right \(G\)-sets
\[ \begin{aligned}{\left({\mu, \varphi}\right) }  \colon & {\Xob} \longrightarrow {\Ymo} \\ & {(a,x)}  \longrightarrow {(\mu(a), \varphi(x)),}\end{aligned}
\qquad
\begin{aligned}{\left({\lambda, \varphi}\right) }  \colon & {\Xob} \longrightarrow {\Ymo} \\ & {(a,x)}  \longrightarrow {(\lambda(x), \varphi(x))}\end{aligned} 
\]
and
\[  \begin{aligned}{\left({\lambda \mu, \varphi}\right) }  \colon & {(F, \varphi)} \longrightarrow {(Q, \varphi)} \\ & {(a,x)}  \longrightarrow {\left({(\lambda \mu)(a), \varphi(x)}\right) .}\end{aligned} 
\]
Then \(\left({\lambda \mu, \varphi}\right) = \left({\lambda, \varphi}\right) \left({\mu, \varphi}\right)\) and \(\left({\id{F}, \varphi }\right)= \id{(F, \varphi)}\).
\end{proposition}
\begin{proof}
We consider the \(2\)-morphism \(\left({\lambda, \varphi}\right)\left({\mu, \varphi}\right) \colon (F,\varphi) \longrightarrow (Q, \varphi)\) given by a morphism in right \({G}\)-set
\[ \begin{aligned}{r}  \colon & {\Xob} \longrightarrow {\Ymo} \\ & {(a, x)}  \longrightarrow {r(a, x) =  m\ped{\Cat{Y} }\Big( \left({\lambda, \varphi}\right) (a, x), \left({\mu, \varphi}\right)(a, x) \Big) . }\end{aligned} 
\]
For each \((a, x) \in \Xob\), we obtain
\[ \begin{aligned}
r(a, x) &= m\ped{\Cat{Y} }\left( \left({\lambda, \varphi}\right)(a, x), \left({\mu, \varphi}\right)(a, x) \right) 
= m\ped{\Cat{Y} }\left( \left({\lambda(a), \varphi(x)}\right), \left({\mu(a), \varphi(x)}\right) \right) \\
&= \left( m\ped{\Cat{C} }\left({\lambda(a), \mu(a)}\right), \varphi(x) \right)
= \left( (\lambda \mu)(a), \varphi(x) \right)
= \left({\lambda \mu, \varphi}\right)(a, x).
\end{aligned}
\]
The \(2\)-morphism \(\id{(F, \varphi)}\) is given by the morphism of right \({G}\)-sets
\[ \begin{aligned}{z}  \colon & {\Xob} \longrightarrow {\Ymo} \\ & {(a,x)}  \longrightarrow {\iota\ped{\Cat{Y} }\left({F(a), \varphi(x)}\right) = \left( \iota\ped{\Cat{D} }F(a), \varphi(x) \right)}\end{aligned} 
\]
and \(\left({\id{F}, \varphi}\right)\) is given by
\[ \begin{aligned}{\left({\id{F}, \varphi }\right) }  \colon & {\Xob} \longrightarrow {\Ymo} \\ & {(a,x)}  \longrightarrow {\left({\id{F}(a), \varphi(x)}\right)= \left({\iota\ped{\Cat{D} } (F(a)), \varphi(x)}\right). }\end{aligned} 
\]
Thus, we get \(\id{(F, \varphi)} = \left({\id{F}, \varphi }\right)\).
\end{proof}

Consider now \(F \colon \Cat{C} \longrightarrow \Cat{D}\) and \(P \colon \Cat{D} \longrightarrow \Cat{C}\) two functors between small categories. Assume we have  \(\varepsilon \colon FP \longrightarrow \id{\Cat{D} }\) and \(\eta \colon PF \longrightarrow \id{\Cat{C}}\) two natural isomorphisms and \(\varphi \colon X \longrightarrow Y\) an isomorphism of  \(G\)-sets. Keeping the previous notations,  we set \(\Xmoi = \Cat{C}\sped{i} \times X\) and \(\Ymoi = \Cat{D}\sped{i} \times Y\) for \(i=0,1\) and consider the morphisms of categorified \(G\)-sets 
\((F, \varphi) \colon \Cat{X} \longrightarrow \Cat{Y}\) and \((P,  \varphi) \colon \Cat{Y} \longrightarrow \Cat{X}\).
Henceforth, we have that 
\[ \left({P, \varphi^{-1} }\right) \left({F, \varphi}\right) = \left({PF, \varphi^{-1} \varphi }\right)= \left({PF, \id{X} }\right)   , \quad  \id{\Cat{X} } = \left({\id{\Cat{C} }, \id{X} }\right)  \colon  \Cat{X} \longrightarrow \Cat{X} 
\]
and
\[ \left({F, \varphi}\right) \left({P, \varphi^{-1}}\right) = \left({FP, \varphi \varphi^{-1}}\right) = \left({FP, \id{Y} }\right)   , \quad  \id{\Cat{Y} } = \left({\id{\Cat{D} }, \id{Y} }\right)    \colon  \Cat{Y} \longrightarrow \Cat{Y} .
\]
We define \(2\)-morphisms
\[ \left({\varepsilon, \id{X} }\right) \colon (P, \psi)(F, \varphi) \longrightarrow \id{\Cat{X} }
\qquad \text{and} \qquad
\left({\eta, \id{Y} }\right) \colon (F, \varphi)(P, \psi) \longrightarrow \id{\Cat{Y} } 
\]
in \(\CRset{G}\) as the morphisms of  \({G}\)-sets
\[ \begin{aligned}{\left({\varepsilon, \id{X} }\right)}  \colon & {\Xob} \longrightarrow {\Xmo} \\ & {(a, x)}  \longrightarrow {(\varepsilon(a), x)}\end{aligned} 
\qquad \text{and} \qquad
\begin{aligned}{\left({\eta, \id{Y} }\right) }  \colon & {\Yob} \longrightarrow {\Ymo} \\ & {(\beta, y)}  \longrightarrow {(\eta(\beta), y)}\end{aligned} 
\]
respectively.
We have already proved that \(\left({\varepsilon, \id{X} }\right)\) and \(\left({\eta, \id{Y} }\right)\) are well defined.
Considering \(\varepsilon^{-1} \colon \id{\Cat{D} } \longrightarrow FG\) and \(\eta^{-1} \colon \id{\Cat{C} } \longrightarrow GF\) we can construct the \(2\)-morphisms
\[ \left({\varepsilon^{-1}, \id{X} }\right) \colon \id{\Cat{X} }  \longrightarrow (P, \psi)(F, \varphi)
\qquad \text{and} \qquad
\left({\eta^{-1}, \id{Y} }\right) \colon  \id{\Cat{Y} }  \longrightarrow (F, \varphi)(P, \psi) 
\]
in \(\CRset{G}\)
\[ \begin{aligned}{ \left({\varepsilon^{-1}, \id{X} }\right) }  \colon & {\Xob} \longrightarrow {\Xmo} \\ & {(a, x)}  \longrightarrow {(\varepsilon^{-1}(a), x)}\end{aligned} 
\qquad \text{and} \qquad
\begin{aligned}{\left({\eta^{-1}, \id{Y} }\right) }  \colon & {\Yob} \longrightarrow {\Ymo} \\ & {(\beta, y)}  \longrightarrow {(\eta^{-1}(\beta), y)}.\end{aligned} 
\]
We calculate
\[ \id{\left( FP, \id{Y} \right) }  = \left({\id{FP} , \id{Y} }\right) 
= \left({\varepsilon^{-1} \varepsilon, \id{Y} }\right) 
= \left({\varepsilon^{-1}, \id{Y} }\right)  \left({\varepsilon, \id{Y} }\right)
\]
and
\[ \id{\left( \id{\Cat{D} }  , \id{Y} \right) }  = \left({\id{ \id{\Cat{D} } } , \id{Y} }\right) 
= \left({\varepsilon \varepsilon^{-1}, \id{Y} }\right) 
= \left({\varepsilon, \id{Y} }\right)  \left({\varepsilon^{-1}, \id{Y} }\right).
\]
In the same way we prove that
\[ \id{\left( PF, \id{X} \right) } =\left({\eta^{-1}, \id{X} }\right)  \left({\eta, \id{X} }\right)
\qquad \text{and} \qquad
\id{\left({\id{\Cat{C} } , \id{X} }\right) } = \left({\eta, \id{X}}\right)  \left({\eta^{-1}, \id{X} }\right).
\]
Therefore, we have that \(\Cat{X} \wea \Cat{Y}\).

Using the forgetful functor of Remark~\(\ref{lOrdinaryCatTheory}\), we deduce, from the previous argumentations,  that \(\Cat{X}\) and \(\Cat{Y}\) cannot be weakly equivalent if the categories \(\Cat{C}\) and \(\Cat{D}\) are not equivalent. Furthermore, we proved the following result.

\begin{proposition}\label{prop:ClassWE}
Let $G$ be any group. Then we have a functor 
$$
\Sf{SCats} \times \rset{G} \longrightarrow \CRset{G}, \quad  \Big( (\cC,X) \longrightarrow \cC \times X   \Big),
$$
where $\Sf{SCats}$ denotes the category of small categories. Moreover, two categorified $G$-sets of the form $\cC\times X$ and $\cD\times Y$ are weakly equivalent if and only if $X$ and $Y$ are isomorphic two $G$-sets and $\cC$, $\cD$ are equivalent categories. 
\end{proposition}
\begin{proof}
Straightforward.
\end{proof}

\section{The right double translation category  of a right categorified group-set}\label{sec:TDCats}

Given a group \(G\), let us consider a \(G\)-set \(X\). It is well known (see, for instance, ~\cite{Kaoutit/Spinosa:2018}) that the pair $(X \times G, X)$ admits in a natural way a structure of a groupoid\footnote{This is by definition a small category where each morphism is an isomorphism.}, called  a \imp{right translation groupoid} \(X \rtimes G\). 
Clearly, the nerve of the underlying category creates a categorified set, which is not a categorified group-set, except in some trivial cases.
This translation groupoid illustrates, in fact, the orbits that the group \(G\) creates acting on \(X\), and its isotropy groups coincide with the stabilizers.
The aim of this section is to explore this  construction for the new  categorified \(G\)-sets.
To this end, the notion of double category, introduced for the first time in~\cite{EhresmannCatStruc}, is essential.

\begin{definition}
A \imp{double category} is an internal category in the category of small categories.
\end{definition}

Given a double category \(\Cat{D}\), the relevant functors are illustrated in the following diagram:
\[ \xymatrix{\Cat{D}\sped{1}  \ar@<1ex>[rr]^{\mathsf{T}\ped{\Cat{D} } } \ar@<-1ex>[rr]_{\mathsf{S}\ped{\Cat{D} } } && \Cat{D}\sped{0}  \ar[rr]^{\mathsf{I}\ped{\Cat{D} } } && \Cat{D}\sped{1} && \Cat{D}\sped{2} \ar[ll]_{\mathsf{M}\ped{\Cat{D} } }  }
\]
where \(\Cat{D}\sped{0}\), \(\Cat{D}\sped{1}\) and \(\Cat{D}\sped{2} = \Cat{D} \pdue{\times}{\mathsf{S}\ped{\Cat{D} } }{\mathsf{T}\ped{\Cat{D} } } \Cat{D}\sped{1}\) are categories (we remind to the reader that the category of small categories has pull-backs).

\begin{definition}
\label{dVertHorMorSquares}
Given a double category \(\Cat{D}\), the set \(\left({\Cat{D}\sped{0}}\right)\sped{0}\) is called the \imp{set of objects}, the set \(\left({\Cat{D}\sped{0}}\right)\sped{1}\) is called the \imp{set of horizontal morphisms}, the set \(\left({\Cat{D}\sped{1}}\right)\sped{0}\) is called the \imp{set of vertical morphisms} and the set \(\left({\Cat{D}\sped{1}}\right)\sped{1}\) is called the \imp{set of squares}.
Moreover, the category \(\Cat{D}\sped{0}\) is called the \imp{category of objects} and \(\Cat{D}\sped{1}\) is called the \imp{category of morphisms}.
\end{definition}

The reason behind Definition~\(\ref{dVertHorMorSquares}\) will be manifest in the forthcoming diagrams of this section that will also illustrate how to operate with double categories.

Given a group \(G\) and a  categorified \(G\)-set \(\Cat{X}\), we set \(\Cat{D}\sped{i} = \Xmoi \rtimes G\) for \(i=0,1\).
We are now going to construct a structure of double category \(\Cat{D}\) starting from the categories \(\Cat{D}\sped{0}\) and \(\Cat{D}\sped{1}\).
We define the \imp{target functor} \(\mathsf{T}\ped{\Cat{D}} \colon \Cat{D}\sped{1} \longrightarrow \Cat{D}\sped{0}\) as in the following diagrams:
\[ \left( \vcenter{  \xymatrix{ x  \ar[d]^{f} \\ y } }  \right)  
\xymatrix@C=3.5em{   \ar[r]^{\left({\mathsf{T}\ped{\Cat{D} } }\right)\sped{0} } &   }  y  
\qquad \text{and} \qquad
\left(  \vcenter{ \xymatrix@R=2.5ex@C=3.5em{ x  \ar[dd]_{f} &  xg  \ar[dd]^{fg} \ar[l]_{(x,g)} \\
   &  \ar@{=>}[l]_{(f,g)}  \\
y  &  yg \ar[l]_{(y,g)} } } \right)  
\xymatrix@C=3.5em{   \ar[r]^{\left({\mathsf{T}\ped{\Cat{D} } }\right)\sped{1} } &   } \left(  \xymatrix@C=3.5em{ y  &  yg \ar[l]_{(y,g)} } \right).
\]
We will now check that \(\mathsf{T}\ped{\Cat{D}}\) is a functor.
Given a vertical morphism \(f \colon x \longrightarrow y\) in \(\left({\Cat{D}\sped{1}}\right)\sped{0}\) we have
\[ \mathsf{T}\ped{\Cat{D}} \left(  \vcenter{ \xymatrix@R=2.5ex@C=3.5em{ x  \ar[dd]_{f} &  xg  \ar[dd]^{fg} \ar[l]_{(x,g)} \\
   &  \ar@{=>}[l]_{(f,g)}  \\
y   & yg \ar[l]_{(y,g)} } } \right)  = \left(  \xymatrix@C=3.5em{ y  &  yg \ar[l]_{(y,g)} } \right)
\]
and, regarding the composition, we compute as follows
\[ 
\begin{aligned}
& \mathsf{T}\ped{\Cat{D}} \left(  \vcenter{ \xymatrix@R=2.5ex@C=3.5em{ x  \ar[dd]_{f}  & xg  \ar[dd]^{fg} \ar[l]_{(x,g)} \\
   &  \ar@{=>}[l]_{(f,g)}  \\
y   & yg \ar[l]_{(y,g)} } } \circ \vcenter{ \xymatrix@R=2.5ex@C=3.5em{ xg  \ar[dd]_{fg}  & xgh  \ar[dd]^{fgh} \ar[l]_{(xg,h)} \\
   &  \ar@{=>}[l]_{(fg,h)}  \\
yg  & ygh \ar[l]_{(yg,h)} } }  \right)  =  \mathsf{T}\ped{\Cat{D}} \left(  \vcenter{ \xymatrix@R=2.5ex{ x  \ar[dd]_{f} &  xgh  \ar[dd]^{fgh} \ar[l]_{(x,gh)} \\
   &  \ar@{=>}[l]_{(f,gh)}  \\
y  &  ygh \ar[l]_{(y,gh)} } }  \right) \\
&= \left(  \xymatrix@C=3.5em{ y  &  ygh \ar[l]_{(y,gh)} } \right) 
= \left(  \xymatrix@C=3.5em{ y  &  yg \ar[l]_{(y,g)} } \right) \circ  \left(  \xymatrix@C=3.5em{ yg &  ygh \ar[l]_{(y,gh)} } \right)  \\
&= \mathsf{T}\ped{\Cat{D}} \left(  \vcenter{ \xymatrix@R=2.5ex{ x  \ar[dd]_{f} &  xgh  \ar[dd]^{fg} \ar[l]_{(x,g)} \\
   &  \ar@{=>}[l]_{(f,gh)}  \\
y  &  ygh\ar[l]_{(y,g)} } }  \right) \mathsf{T}\ped{\Cat{D}} \left(  \vcenter{ \xymatrix@R=2.5ex{ xg  \ar[dd]_{fg} &  xgh  \ar[dd]^{fgh} \ar[l]_{(xg,h)} \\
   &  \ar@{=>}[l]_{(fg,h)}  \\
yg  &  ygh \ar[l]_{(yg,h)} } }  \right).
\end{aligned}
\]
We define the \imp{source functor} \(\mathsf{S}\ped{\Cat{D}} \colon \Cat{D}\sped{1} \longrightarrow \Cat{D}\sped{0}\) as in the following diagrams:
\[ \left( \vcenter{  \xymatrix{ x  \ar[d]^{f} \\ y } }  \right)  
\xymatrix@C=3.5em{   \ar[r]^{\left({\mathsf{S}\ped{\Cat{D} } }\right)\sped{0} } &   }  x  
\qquad \text{and} \qquad
\left(  \vcenter{ \xymatrix@R=2.5ex@C=3.5em{ x  \ar[dd]_{f} &  xg  \ar[dd]^{fg} \ar[l]_{(x,g)} \\
   &  \ar@{=>}[l]_{(f,g)}  \\
y  &  yg \ar[l]_{(y,g)} } } \right)  
\xymatrix@C=3.5em{   \ar[r]^{\left({\mathsf{S}\ped{\Cat{D} } }\right)\sped{1} } &   } \left(  \xymatrix@C=3.5em{ x  &  xg \ar[l]_{(x,g)} } \right).
\]
The \imp{identity functor} \(\mathsf{I}\ped{\Cat{D}} \colon \Cat{D}\sped{0} \longrightarrow \Cat{D}\sped{1}\) can be defined as in the following diagrams:
\[ x \xymatrix@C=3.5em{   \ar[r]^{\left({\mathsf{I}\ped{\Cat{D} } }\right)\sped{0} } &   }  
\left( \vcenter{  \xymatrix{ x  \ar[d]^{\id{x} } \\ x } }  \right)  
\qquad \text{and} \qquad
\left(  \xymatrix@C=3.5em{ x  &  xg \ar[l]_{(x,g)} } \right)
\xymatrix@C=3.5em{   \ar[r]^{\left({\mathsf{I}\ped{\Cat{D} } }\right)\sped{1} } &   }   
\left(  \vcenter{ \xymatrix@R=2.5ex@C=3.5em{ x  \ar[dd]_{\id{x} } &  xg  \ar[dd]^{\id{x} g} \ar[l]_{(x,g)} \\
   &  \ar@{=>}[l]_{(\id{x} ,g)}  \\
x  &  xg \ar[l]_{(x,g)} } } \right).
\]
The proof that \(\mathsf{S}\ped{\Cat{D}}\) and \(\mathsf{I}\ped{\Cat{D}}\) are functors is similar to the one of \(\mathsf{T}\ped{\Cat{D}}\).
Regarding the multiplication functor \(\mathsf{M}\ped{\Cat{D}} \colon \Cat{D}\sped{2} \longrightarrow \Cat{D}\sped{1}\), we define
\[ \left( \vcenter{  \xymatrix{ y  \ar[d]^{l} \\ z } }, \vcenter{  \xymatrix{ x  \ar[d]^{f} \\ y } }  \right)  \xymatrix@C=3.5em{   \ar[r]^{\left({\mathsf{M}\ped{\Cat{D} } }\right)\sped{0} } &   }  
\left( \vcenter{  \xymatrix{ x  \ar[d]^{lf} \\ z } }  \right),\, 
\quad
 \left(  \vcenter{ \xymatrix@R=2.5ex@C=3.5em{ y  \ar[dd]_{l } &  yg  \ar[dd]^{l g} \ar[l]_{(y,g)} \\
   &  \ar@{=>}[l]_{(l ,g)}  \\
z  &  zg \ar[l]_{(z,g)} } } , \vcenter{ \xymatrix@R=2.5ex@C=3.5em{ x  \ar[dd]_{f } &  xg  \ar[dd]^{f g} \ar[l]_{(x,g)} \\
   &  \ar@{=>}[l]_{(f ,g)}  \\
y  &  yg \ar[l]_{(y,g)} } } \right)
\xymatrix@C=3.5em{   \ar[r]^{\left({\mathsf{M}\ped{\Cat{D} } }\right)\sped{1} } &   }   
\left(  \vcenter{ \xymatrix@R=2.5ex@C=3.5em{ x  \ar[dd]_{lf } &  xg  \ar[dd]^{(lf) g} \ar[l]_{(x,g)} \\
   &  \ar@{=>}[l]_{(lf ,g)}  \\
z  &  zg \ar[l]_{(z,g)} } } \right).
\]

The fact that  \(\mathsf{M}\ped{\Cat{D}}\) is a functor, follows  from the subsequent calculations:
\[ 
\mathsf{M}\ped{\Cat{D}} \left(  \vcenter{ \xymatrix@R=2.5ex@C=3.5em{ y  \ar[dd]_{l} &  y  \ar[dd]^{l} \ar[l]_{(y,1)} \\
   &  \ar@{=>}[l]_{(l,1)}  \\
z   & z \ar[l]_{(z,1)} } },  \vcenter{ \xymatrix@R=2.5ex@C=3.5em{ x  \ar[dd]_{f} &  x  \ar[dd]^{f} \ar[l]_{(x,1)} \\
   &  \ar@{=>}[l]_{(f,1)}  \\
y   & y \ar[l]_{(y,1)} } } \right) = \left(  \vcenter{ \xymatrix@R=2.5ex@C=3.5em{ x  \ar[dd]_{lf} &  x  \ar[dd]^{lf} \ar[l]_{(x,1)} \\
   &  \ar@{=>}[l]_{(lf,1)}  \\
z   & z \ar[l]_{(z,1)} } } \right) 
\]
and, since \((lf)g = (lg)(fg)\), for every $g \in G$, we have that
\[ \begin{aligned}
& \mathsf{M}\ped{\Cat{D}} \left(   \left(  \vcenter{ \xymatrix@R=2.5ex@C=3.5em{ y  \ar[dd]_{l } &  yg  \ar[dd]^{l g} \ar[l]_{(y,g)} \\
   &  \ar@{=>}[l]_{(l ,g)}  \\
z  &  zg \ar[l]_{(z,g)} } } , \vcenter{ \xymatrix@R=2.5ex@C=3.5em{ x  \ar[dd]_{f } &  xg  \ar[dd]^{f g} \ar[l]_{(x,g)} \\
   &  \ar@{=>}[l]_{(f ,g)}  \\
y  &  yg \ar[l]_{(y,g)} } } \right) \circ  \left(  \vcenter{ \xymatrix@R=2.5ex@C=3.5em{ yg  \ar[dd]_{lg } &  ygh  \ar[dd]^{l gh} \ar[l]_{(yg,h)} \\
   &  \ar@{=>}[l]_{(lg ,h)}  \\
zg  &  zgh \ar[l]_{(zg,h)} } } , \vcenter{ \xymatrix@R=2.5ex@C=3.5em{ xg  \ar[dd]_{fg } &  xgh  \ar[dd]^{f gh} \ar[l]_{(xg,h)} \\
   &  \ar@{=>}[l]_{(fg ,h)}  \\
yg  &  ygh \ar[l]_{(yg,h)} } } \right) \right) \\
&= \mathsf{M}\ped{\Cat{D}} \left(  \vcenter{ \xymatrix@R=2.5ex@C=3.5em{ y  \ar[dd]_{l } &  yg  \ar[dd]^{l g} \ar[l]_{(y,g)} \\
   &  \ar@{=>}[l]_{(l ,g)}  \\
z  &  zg \ar[l]_{(z,g)} } } \circ \vcenter{ \xymatrix@R=2.5ex@C=3.5em{ yg  \ar[dd]_{lg } &  ygh  \ar[dd]^{l gh} \ar[l]_{(yg,h)} \\
   &  \ar@{=>}[l]_{(lg ,h)}  \\
zg  &  zgh \ar[l]_{(zg,h)} } } , \vcenter{ \xymatrix@R=2.5ex@C=3.5em{ x  \ar[dd]_{f } &  xg  \ar[dd]^{f g} \ar[l]_{(x,g)} \\
   &  \ar@{=>}[l]_{(f ,g)}  \\
y  &  yg \ar[l]_{(y,g)} } }   \circ \vcenter{ \xymatrix@R=2.5ex@C=3.5em{ xg  \ar[dd]_{fg } &  xgh  \ar[dd]^{f gh} \ar[l]_{(xg,h)} \\
   &  \ar@{=>}[l]_{(fg ,h)}  \\
yg  &  ygh \ar[l]_{(yg,h)} } }  \right) \\
&= \mathsf{M}\ped{\Cat{D}} \left(  \vcenter{ \xymatrix@R=2.5ex@C=3.5em{ y  \ar[dd]_{l } &  ygh  \ar[dd]^{l gh} \ar[l]_{(y,gh)} \\
   &  \ar@{=>}[l]_{(l ,gh)}  \\
z  &  zgh \ar[l]_{(z,gh)} } }  , \vcenter{ \xymatrix@R=2.5ex@C=3.5em{ x  \ar[dd]_{f } &  xgh  \ar[dd]^{f gh} \ar[l]_{(x,gh)} \\
   &  \ar@{=>}[l]_{(f ,gh)}  \\
y  &  ygh \ar[l]_{(y,gh)} } }     \right)
= \left( \vcenter{ \xymatrix@R=2.5ex@C=3.5em{ x  \ar[dd]_{lf } &  xgh  \ar[dd]^{(lf) gh} \ar[l]_{(x,gh)} \\
   &  \ar@{=>}[l]_{(lf ,gh)}  \\
z  &  zgh \ar[l]_{(z,gh)} } } \right) \\
&= \left( \vcenter{ \xymatrix@R=2.5ex@C=3.5em{ x  \ar[dd]_{lf } &  xg  \ar[dd]^{(lf) g} \ar[l]_{(x,g)} \\
   &  \ar@{=>}[l]_{(lf ,g)}  \\
z  &  zg \ar[l]_{(z,g)} } } \right) \circ \left( \vcenter{ \xymatrix@R=2.5ex@C=3.5em{ xg  \ar[dd]_{(lf)g } &  xgh  \ar[dd]^{(lf) gh} \ar[l]_{(xg,h)} \\
   &  \ar@{=>}[l]_{((lf)g,h)}  \\
zg  &  zgh \ar[l]_{(zg,h)} } } \right) \\
&= \mathsf{M}\ped{\Cat{D}} \left(  \vcenter{ \xymatrix@R=2.5ex@C=3.5em{ y  \ar[dd]_{l } &  yg  \ar[dd]^{l g} \ar[l]_{(y,g)} \\
   &  \ar@{=>}[l]_{(l ,g)}  \\
z  &  zg \ar[l]_{(z,g)} } }  , \vcenter{ \xymatrix@R=2.5ex@C=3.5em{ x  \ar[dd]_{f } &  xg  \ar[dd]^{f g} \ar[l]_{(x,g)} \\
   &  \ar@{=>}[l]_{(f ,g)}  \\
y  &  yg \ar[l]_{(y,g)} } }     \right)
\circ \mathsf{M}\ped{\Cat{D}} \left(  \vcenter{ \xymatrix@R=2.5ex@C=3.5em{ yg  \ar[dd]_{lg } &  ygh  \ar[dd]^{l gh} \ar[l]_{(yg,h)} \\
   &  \ar@{=>}[l]_{(lg ,h)}  \\
zg  &  zgh \ar[l]_{(zg,h)} } }  , \vcenter{ \xymatrix@R=2.5ex@C=3.5em{ xg  \ar[dd]_{fg } &  xgh  \ar[dd]^{f gh} \ar[l]_{(xg,h)} \\
   &  \ar@{=>}[l]_{(fg ,h)}  \\
yg  &  ygh \ar[l]_{(yg,h)} } }     \right),
\end{aligned}
\]
for every $h \in G$.

\begin{definition}\label{def:TDC}
The double category \(\Cat{D}\) just constructed is called the \imp{right translation double category} of the right categorified \(G\)-set \(\Cat{X}\) and we denote it by \(\Cat{X } \rtimes G\).
\end{definition}

\begin{example}\label{exam:TDC}
Given a small category \(\Cat{C}\) and a group \(G\), let \(\Cat{X}= \Cat{C} \times G\) be the right categorified \(G\)-set of section~\ref{secExCatGSets}.
We want to describe the double translation category \(\Cat{D} = \Cat{C} \times X\).
For \(i=0,1\) we have \(\left({\Cat{D}\sped{0}}\right)\sped{0} = \left({\Xob \rtimes G}\right)\sped{0} = \Xob\), \(\left({\Cat{D}\sped{0} }\right)\sped{1} = \left({\Xob \times G}\right)\sped{1} = \Xob \times G\), \(\left({\Cat{D}\sped{1}}\right)\sped{0} = \left({\Xmo \rtimes G}\right)\sped{0} = \Xmo\) and \(\left({\Cat{D}\sped{1} }\right)\sped{1} = \left({\Xmo \times G}\right)\sped{1} = \Xmo \times G\).
The target, source and identity functors are as follows:
\[ \mathsf{T}\ped{\Cat{D} }= \left({\mathsf{t}\ped{\Cat{C} } \times \id{X}, \mathsf{t}\ped{\Cat{C} } }\right),
\qquad 
\mathsf{S}\ped{\Cat{D} }= \left({\mathsf{s}\ped{\Cat{C} } \times \id{X}, \mathsf{s}\ped{\Cat{C} } }\right)
\qquad \text{and} \qquad
\mathsf{I}\ped{\Cat{D} }= \left({\iota\ped{\Cat{C} } \times \id{X}, \iota\ped{\Cat{C} } }\right),
\]
Regarding the composition functor, we have:
\[ \begin{aligned}
 \mathsf{M}\ped{\Cat{D} } & \colon \Cat{D}\sped{2} \longrightarrow \Cat{D}\sped{1} \\
& \left({\Cat{D}\sped{2} }\right)\sped{0} \ni  (h,f) \longrightarrow m\ped{\Cat{X} }(h,f) \\
& \left({\Cat{D}\sped{2} }\right)\sped{1} \ni \left({(h,g),(f,g)}\right) \longrightarrow \left({m\ped{\Cat{X} }(h,f),g}\right) .
\end{aligned}
\]
\end{example}

\begin{definition}
Given a double category \(\Cat{D}\), let us consider a square \(Q \in \left({\Cat{D}\sped{1} }\right)\sped{1}\).
We define the \imp{vertices of the square \(Q\)} as \(\mathsf{s}\ped{\Cat{D}\sped{0}}\left(\mathsf{S}\ped{\Cat{D}}(Q)\right)\), \(\mathsf{t}\ped{\Cat{D}\sped{0}}\left(\mathsf{S}\ped{\Cat{D}}(Q)\right)\), \(\mathsf{s}\ped{\Cat{D}\sped{0}}\left(\mathsf{T}\ped{\Cat{D}}(Q)\right)\) and \(\mathsf{t}\ped{\Cat{D}\sped{0}}\left(\mathsf{T}\ped{\Cat{D}}(Q)\right)\).
\end{definition}

\begin{example}
\label{exSquare}
Given \(\Cat{X} \in \CRset{G}\), in the following square in \(\Cat{X} \rtimes G\)
\[ \xymatrix@R=2.5ex@C=3.5em{ x  \ar[dd]_{f} &  xg  \ar[dd]^{fg} \ar[l]_{(x,g)} \\
   &  \ar@{=>}[l]_{(f,g)}  \\
y  &  yg \ar[l]_{(y,g)} }
\]
the vertices are \(x\), \(xg\), \(y\) and \(yg\).
\end{example}

\begin{definition}
Let be \(\Cat{X} \in \CRset{G}\): for each \(a, b \in \Xob\) we define the relation  \(a \, \sqre \, b\) if and only if there is a square of \(\Cat{X} \rtimes G\) with \(a\) and \(b\) amongst its vertices.
\end{definition}

\begin{remark}
\label{rSquare}
Given a morphism \(\left(f \colon a \longrightarrow b\right) \in \Xmo\), an object \(\alpha \in \Xob\) and \(g \in G\) we have the diagrams
\begin{center}
\begin{minipage}[c]{0.30\textwidth}
\centering
\[ \xymatrix@R=2.5ex@C=3.5em{ a  \ar[dd]_{f} &  a  \ar[dd]^{f} \ar[l]_{(a,1)} \\
   &  \ar@{=>}[l]_{(f,1)}  \\
b  &  b, \ar[l]_{(b,1)} }
\]
\end{minipage}%
\begin{minipage}[c]{0.10\textwidth}
\centering
and
\end{minipage}%
\begin{minipage}[c]{0.30\textwidth}
\centering
\[ \xymatrix@R=2.5ex@C=3.5em{ \alpha  \ar[dd]_{\id{\alpha} } &  \alpha g  \ar[dd]^{\id{\alpha}g = \id{\alpha g} } \ar[l]_{(\alpha,g)} \\
   &  \ar@{=>}[l]_{\left({\id{\alpha }, g}\right)}  \\
\alpha  &  \alpha g, \ar[l]_{(\alpha,g)} }
\]
\end{minipage}%
\end{center}
therefore \(a \, \sqre \, b\) and \(\alpha \, \sqre \, \alpha g\).
\end{remark}

\begin{remark}
The relation \(\sqre\) is reflexive and symmetric but not transitive: it's enough to consider the following example with a trivial action: \(\Xob= \Set{a, b, c}\) and only \(f \colon  a \longrightarrow b\) and \(h \colon a \longrightarrow c\) as not isomorphisms arrows.
In this case it is clear that \(a \, \sqre \, b\) and \(a \, \sqre \, c\) but we don't have \(b \, \sqre \, c\).
\end{remark}

This suggests us to give the next definition.

\begin{definition}
Let be \(\Cat{X} \in \CRset{G}\): for each \(a, b \in \Xob\) we define \(a \, \Sqre \, b\) if and only if there are \(a_0, \dots, a_n \in \Xob\), with \(n \in \mathbb{N}^+\), such that for each \({i}\in \Set{{0},\dots,{n-1}}\), \(a_i \, \sqre \, a\ped{i+1}\).
This means that \(\Sqre\) is the equivalence relations generated by \(\sqre\).
Given \(a \in \Xob\), we denote with \(\SqOrb{a}\) the full subcategory of \(\Cat{X}\) such that the set of objects of \(\SqOrb{a}\) is the equivalence class of \(a\) with respect to \(\Sqre\).
We also set \(\SqOrb{f}:= \SqOrb{\mathsf{s}\ped{\Cat{X}}(f)} = \SqOrb{\mathsf{t}\ped{\Cat{X}}(f)}\) for every \(f \in \Xmo\).
Moreover, we denote with \(\SqRep{\Cat{X}}\) a set of objects of \(\Cat{X}\) that acts as a set of representative elements with respect to the relation \(\Sqre\).
Note that, for every \(f \in \Xmo\) (respectively, for each \(a \in \Xob\)), \(\SqOrb{f}\) (respectively, \(\SqOrb{a}\)) contains both the \(G\)-orbit of \(f\) (respectively, of \(a\)) and the connected component of the category \(\Cat{X}\) (see~\cite[pages 88 and 90]{Lane1998}) that contains \(f\) (respectively, \(a\)).
As a consequence, both \(\SqOrb{f}\) and \(\SqOrb{a}\) are right \(G\)-sets,  can be decomposed in \(G\)-orbits and are called the \imp{orbit categories} of \(f\) and \(a\), respectively.
\end{definition}

\begin{remark}
\label{rSqOrbSubCat}
Given a categorified \(G\)-set \(\Cat{X}\), let us consider a categorified \(G\)-subset \(\Cat{Y}\) of \(\Cat{X}\).
If we consider a square like the one in Example~\(\ref{exSquare}\), with \(\left( f \colon a \longrightarrow b\right) \in \Xmo\) and \(g \in G\), then, thanks to Remark~\(\ref{rSquare}\), the following statements are equivalent:
\begin{enumerate}
\item \(x \in \Yob\);
\item \(xg \in \Yob\);
\item \(yg \in \Yob\);
\item \(y \in \Yob\).
\end{enumerate}
As a consequence, for each \(\alpha \in \Yob\), \(\SqOrb{\alpha}\) is a categorified \(G\)-subset of \(\Cat{Y}\).
\end{remark}

\begin{proposition}
Let be \(\Cat{X} \in \CRset{G}\): we have
\[ \Cat{X} = \biguplus\ped{a \, \in \, \SqRep{\Cat{X}}} \SqOrb{a}
\]
and, for every \(a \, \in \, \SqRep{\Cat{X}}\), there cannot be two categorified \(G\)-sets \(\Cat{Y}\) and \(\Cat{Y}'\), both not empty, such that \(\SqOrb{a} = \Cat{Y} \Uplus \Cat{Y}'\).
\end{proposition}
\begin{proof}
Every category is the disjoint union of its connected components and, for each \(a \in \Xob\), \(\SqOrb{a}\) is the disjoint union of the connected components of \(\Cat{X}\) with objects in the \(\Sqre\) relation with \(a\).

Regarding the last statement, by contradiction, let us assume that there are two categorified \(G\)-sets \(\Cat{Y}\) and \(\Cat{Y}'\) such that \(\Yob \cap \Yob' = \emptyset\) and \(\Cat{Y} \neq \emptyset \neq \Cat{Y}'\).
We consider \(b \in \Yob\) and \(b' \in \Yob'\): then \(b , b' \in \SqOrb{a}\) thus \(b \, \Sqre \, a \, \Sqre \, b'\).
Thanks to Remark~\(\ref{rSqOrbSubCat}\) we obtain \(b   \in \left(\SqOrb{b'}\right)_{0} \subseteq \Yob'\), that is absurd because \(b \in \Yob\).
\end{proof}

\section{Burnside ring functors of groups}\label{sec:BGrps}

We will give the main steps to construct the categorified Burnside  ring of a given group, using its category of categorified group-sets, and we will compare this new ring with the classical one, providing a natural transformation between the two contravariant functors.
Moreover, we will show that this natural transformation is injective but not surjective and we will conclude illustrating, with an example, the difference between the classical case and the categorified one.

Given a morphism of groups \(\phi \colon H \longrightarrow G\), we define the induced functor, referred to as the \imp{induction functor}
\[\phi^\ast  \colon \CRset{G} \longrightarrow \CRset{H},
\]
which sends the  categorified \(G\)-set \(\Cat{X}\) to the  categorified \(H\)-set \(\phi^\ast\left({\Cat{X}}\right)\) with the action \(x \cdot h := x \phi(h)\) for each \(x \in \Xmoi\) and \(h \in H\), for \(i=0,1\).
The target, source, identity and multiplication maps remain the same but they will be morphisms in \(\rset{H}\) and not in \(\rset{G}\).
Given a morphism of  categorified \(G\)-set \(f \colon \Cat{X}  \longrightarrow \Cat{Y}\), we define the morphism of categorified  \(H\)-sets \(\phi^\ast(f) \colon \phi^\ast\left(\Cat{X}\right)  \longrightarrow \phi^\ast\left(\Cat{Y}\right)\) as the morphism \(f\).
In a similar way to how it has been done in~\cite{KaoKowMoritaHAlgd} and~\cite[Proposition 3.2]{KaoutitSpinosaBurn}, it is possible to prove that \(\phi^\ast\) is monoidal with respect to both \(\Uplus\) and \(\times\). Thus, $\phi^*$ is Laplaza functor, in the sense of \cite[Appendices]{KaoutitSpinosaBurn}

Let \(G\) be a group.
We will now briefly develop a Burnside theory based on categorified  \(G\)-sets, following the similar picture as it has been done in~\cite{KaoutitSpinosaBurn}.
We recall that by the word \imp{rig} we mean a unital commutative ring that doesn't necessarily have additive inverses (another term often used is semiring, see for instance~\cite[Appendix~A.2]{KaoutitSpinosaBurn}).
We denote by \(\fCRset{G}\) the \(2\)-category of finite  categorified \(G\)-sets (we say that \(\Cat{X} \in \CRset{G}\) is finite if \(\Xmo\) is finite); however, as with \(\CRset{G}\), we will mainly consider it as a category.
We denote by \(\siL\left({G}\right)\) the quotient set of \(\fCRset{G}\) by the equivalence relation \(\wea\).
Given a morphism of groups \(\phi \colon H \longrightarrow G\), thanks to~\cite[Lemma~6.1]{KaoutitSpinosaBurn}, it is clear that we have a monomorphism of rigs
\[ \begin{aligned}{\siL\left({\phi}\right)}  \colon & {\siL\left({G}\right)} \longrightarrow {\siL\left({H}\right)} \\ & {\left[{X}\right] }  \longrightarrow {\left[{\phi^\ast\left({X}\right) }\right] .}\end{aligned} 
\]
In this way we obtain a contravariant functor \(\siL \colon \bf{Grp} \longrightarrow \riG\), that is, from the category of groups to the category of rigs.

Given a group \(G\), using the functor \(\SInc{G}\) of Eq.~\(\eqref{eIncFun}\), we can construct an injective morphism from the classical Burnside rig of \(G\) to the categorified Burnside rig of \(G\) in the following way:
\[ \begin{aligned}{\siIL(G)}  \colon & {\lL\left({G}\right) } \longrightarrow {\siL\left({G}\right) } \\ & {\left[{X}\right] }  \longrightarrow {\left[{\SInc{G}\left({X}\right) }\right] .}\end{aligned} 
\]
In this way we obtain a natural transformation \(\siIL \colon \lL \longrightarrow \siL\), that is, given a morphism of groups \(\phi \colon H \longrightarrow G\), the following diagram is commutative:
\[ \xymatrix{ 
\lL(G) \ar[rr]^{\siIL(G)} \ar[d]_{\lL(\phi)} & & \siL(G)  \ar[d]^{\siL(\phi)}  \\
\lL(H) \ar[rr]^{\siIL(H)} & & \siL(H). 
}
\]

\begin{remark}
\label{rCancProp}
In the classical case the cancellative property of the additive monoid \(\lL\left({G}\right)\) is guaranteed by the Burnside Theorem (see~\cite[Thm.~2.4.5]{BoucBisetFunctors}) but, in this context, we still don't know whether a similar theorem is true or not and, consequently, we cannot say whether the additive monoid \(\siL(G)\) satisfies the cancellative property or not.
\end{remark}

\begin{proposition}
The rig homomorphism \(\siIL(G)\) is injective but not surjective.
\end{proposition}
\begin{proof}
We will use the abuses of notations \(\SInc{G} = \mathcal{I}\) and \(\siIL = \siIL(G)\).
Given \(X, Y \in \Rset{G}\) such that \(\siIL\left(\left[X\right]\right) = \siIL\left(\left[Y\right]\right)\), we have \(\left[\mathcal{I}(X)\right] = \siIL\left(\left[X\right]\right) = \siIL\left(\left[Y\right]\right) = \left[\mathcal{I}(Y)\right]\) thus \(\mathcal{I}(X) \wea \mathcal{I}(Y)\).
Therefore
\[ \mathcal{I}(X) \cong \Ske\left(\mathcal{I}(X)\right) 
\cong \Ske\left(\mathcal{I}(Y)\right) \cong \mathcal{I}(Y)
\]
and \(\left[X\right] = \left[Y\right]\).

Now, by contradiction, let's assume that \(\siIL\) is surjective.
We consider a categorified \(G\)-set \(\Cat{X}\) such that \(\Cat{X} \neq \dCat\left(\Cat{X}\right)\) (or, equivalently, such that \(\ndCat\left(\Cat{X}\right) \neq \emptyset\)).
Then there is a \(G\)-set \(Y\) such that \(\left[\Cat{X}\right] = \siIL\left(\left[Y\right]\right) = \left[ \mathcal{I}\left(Y\right)\right]\) thus \(\Cat{X} \wea \mathcal{I}\left(Y\right)\).
As a consequence
\[ \Ske\left(\Cat{X}\right) \cong \Ske\left(\mathcal{I}\left(Y\right)\right) \cong \mathcal{I}\left(Y\right)
\]
thus \(\ndCat\left(\Cat{\Ske\left(\Cat{X}\right)}\right) \cong \ndCat\left(\mathcal{I}\left(Y\right)\right) = \emptyset\), which is absurd.
\end{proof}

We define \(\siB = \gG \siL\) and \(\siIB = \gG \siIL\), where \({\gG}\) is the Grothendieck functor (see~\cite[Appendix~A.3]{KaoutitSpinosaBurn}), that is, a functor that associates to each rig an opportune ring with a specific universal property.
In this way we obtain a contravariant functor \(\siB \colon \bf{Grp} \longrightarrow \cring\), that is, from the category of groups to the category of commutative rings, and a natural transformation
\[ 
\siIB = \gG \siIL \colon \bB = \gG \lL \longrightarrow \siB = \gG \siL.
\]
The commutative ring $\siB(G)$, for a given group $G$, is called the \emph{categorified Burnside ring} of $G$.

\begin{proposition}
The ring homomorphism \(\siIB(G)\) is injective but not surjective.
\end{proposition}
\begin{proof}
We will use the abuses of notations \(\SInc{G} = \mathcal{I}\), \(\siIL = \siIL(G)\) and \(\siIB = \siIB(G)\).
We consider
\[ \Big[ \left[{X}\right], \left[{Y}\right] \Big], \quad \Big[ \left[{A}\right], \left[{B}\right] \Big] \in \siB(G)
\]
such that
\[ \Big[ \siIL\left({\left[{X}\right] }\right), \siIL\left({\left[{Y}\right] }\right) \Big] = \siIB\left({\Big[ \left[{X}\right], \left[{Y}\right] \Big]  }\right) 
= \siIB\left({\Big[ \left[{A}\right], \left[{B}\right] \Big]  }\right) 
= \Big[ \siIL\left({\left[{A}\right] }\right), \siIL\left({\left[{B}\right] }\right) \Big].
\]
There is \(\left[{\Cat{E} }\right] \in \siL(G)\) such that
\[ \siIL\left({\left[{X}\right] }\right) + \siIL\left({\left[{B}\right] }\right) + \left[{\Cat{E} }\right] = \siIL\left({\left[{A}\right] }\right) + \siIL\left({\left[{Y}\right] }\right) + \left[{\Cat{E} }\right] ,
\]
therefore
\[ \mathcal{I}\left(X\right) \Uplus \mathcal{I}\left(B\right) \Uplus \Cat{E} \wea \mathcal{I}\left(A\right) \Uplus \mathcal{I}\left(Y\right) \Uplus \Cat{E}.
\]
As a consequence we obtain
\[ \begin{aligned}
& \quad \mathcal{I}\left(X\right) \Uplus \mathcal{I}\left(B\right) \Uplus \Ske\left(\Cat{E}\right) \cong \Ske\left(\mathcal{I}\left(X\right)\right) \Uplus \Ske\left(\mathcal{I}\left(B\right)\right) \Uplus \Ske\left(\Cat{E}\right) \\
& \cong \Ske\left(\mathcal{I}\left(A\right)\right) \Uplus \Ske\left(\mathcal{I}\left(Y\right)\right) \Uplus \Ske\left(\Cat{E}\right)
\cong \mathcal{I}\left(A\right) \Uplus \mathcal{I}\left(Y\right) \Uplus \Ske\left(\Cat{E}\right) ,
\end{aligned}
\]
thus
\[ \begin{aligned}
& \quad \mathcal{I}\left(X\right) \Uplus \mathcal{I}\left(B\right) \Uplus \dCat\left(\Ske\left(\Cat{E}\right)\right) \cong \dCat\left(\mathcal{I}\left(X\right)\right) \Uplus \dCat\left(\mathcal{I}\left(B\right)\right) \Uplus \dCat\left(\Ske\left(\Cat{E}\right)\right) \\
& \cong  \dCat\left(\mathcal{I}\left(A\right)\right) \Uplus \dCat\left(\mathcal{I}\left(Y\right)\right) \Uplus \dCat\left(\Ske\left(\Cat{E}\right)\right) \cong \mathcal{I}\left(A\right) \Uplus \mathcal{I}\left(Y\right) \Uplus \dCat\left(\Ske\left(\Cat{E}\right)\right)
\end{aligned}
\]
therefore, considering that \(\dCat\left(\Ske\left(\Cat{E}\right)\right)\) is finite and can be considered a classical \(G\)-set, since the additive monoid \(\lL\left({G}\right)\) satisfies the cancellative property (see Remark~\(\ref{rCancProp}\)), we obtain
\[ \mathcal{I}\left(X\right) \Uplus \mathcal{I}\left(B\right) \cong \mathcal{I}\left(A\right) \Uplus \mathcal{I}\left(Y\right).
\]
This is equivalent to say that \(X \Uplus B \cong A \Uplus Y\), that is, \(\Big[ \left[{X}\right], \left[{Y}\right] \Big] = \Big[ \left[{A}\right], \left[{B}\right] \Big]\).

Now, by contradiction, let's assume that \(\siIB\) is surjective.
We consider a categorified \(G\)-set \(\Cat{X}\) such that \(\Cat{X} \neq \dCat\left(\Cat{X}\right)\) (or, equivalently, such that \(\ndCat\left(\Cat{X}\right) \neq \emptyset\)).
Then there are \(G\)-sets \(Y\) and \(Z\) such that
\[ \Big[ \left[\Cat{X}\right], \left[\emptyset\right] \Big] = \siIB\left(\Big[ \left[Y\right], \left[Z\right] \Big] \right) 
= \Big[ \siIL\left(\left[Y\right]\right), \siIL\left(\left[Z\right]\right)  \Big] ,
\]
thus there is a categorified \(G\)-set \(\Cat{E}\) such that
\[ \left[{\Cat{X} }\right] + \siIL\left({\left[{Z}\right] }\right) + \left[{\Cat{E} }\right] = \siIL\left({\left[{Y}\right] }\right) + \left[\emptyset\right] + \left[{\Cat{E} }\right]  .
\]
As a consequence, \(\Cat{X} \Uplus  \mathcal{I}\left(Z\right) \Uplus  \Cat{E} \wea \mathcal{I}\left(Y\right) \Uplus  \Cat{E}\) thus
\[ \Ske\left(\Cat{X}\right) \Uplus  \Ske\left(\mathcal{I}\left(Z\right)\right) \Uplus  \Ske\left(\Cat{E}\right) \cong \Ske\left(\mathcal{I}\left(Y\right)\right) \Uplus \Ske\left(\Cat{E}\right) .
\]
Considering that \(\ndCat\left( \Ske\left(\mathcal{I}\left(Z\right)\right) \right) = \emptyset\) and \(\ndCat\left( \Ske\left(\mathcal{I}\left(Y\right)\right) \right) = \emptyset\),
we obtain 
\[ \ndCat\left(\Ske\left(\Cat{X}\right)\right) \Uplus \ndCat\left(\Ske\left(\Cat{E}\right)\right) \cong \ndCat\left(\Ske\left(\Cat{E}\right)\right) .
\]
Since both \(\Cat{X}\) and \(\Cat{E}\) are finite categorified \(G\)-sets (\(\Xmo\) and \(\Cat{E}_1\) are finite), examining the cardinalities, we have
\[ \left\vert\ndCat\left(\Ske\left(\Cat{X}\right)\right)\right\vert + \left\vert\ndCat\left(\Ske\left(\Cat{E}\right)\right)\right\vert = \left\vert\ndCat\left(\Ske\left(\Cat{E}\right)\right)\right\vert ,
\]
that implies \(\left\vert\ndCat\left(\Ske\left(\Cat{X}\right)\right)\right\vert =0\).
This means that \(\emptyset =  \ndCat\left(\Ske\left(\Cat{X}\right)\right) \cong \Ske\left(\ndCat\left(\Cat{X}\right)\right)\), therefore \(\ndCat\left(\Cat{X}\right) = \emptyset\), which is absurd.
\end{proof}

\begin{example}\label{exam:ExmGrps}
Given a group \(G\), the following categories, with only the identities as isomorphisms and with the actions opportunely defined, thanks to Proposition~\(\ref{cWeakEqIsom}\),
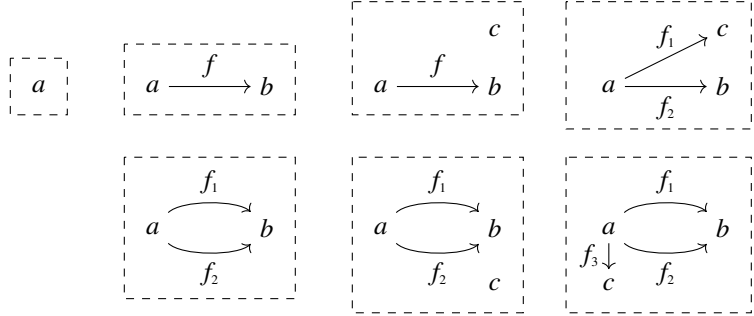
\begin{center}
\begin{tikzpicture}[scale=0.75]

\node at (0,0) (a) {\(a\)};

\node at (2,0) (a2) {\(a\)};
\node at (4,0) (b2) {\(b\)};
\draw [->] (a2) -- (b2) node[midway,above]{\(f\)};

\node at (6,0) (a3) {\(a\)};
\node at (8,0) (b3) {\(b\)};
\draw [->] (a3) -- (b3) node[midway,above]{\(f\)};
\node at (8,1) (c3) {\(c\)};

\node at (10,0) (a4) {\(a\)};
\node at (12,0) (b4) {\(b\)};
\draw [->] (a4) -- (b4) node[midway,below]{\(f_{\Sscript{2}}\)};
\node at (12,1) (c4) {\(c\)};
\draw [->] (a4) -- (c4) node[midway,above]{\(f_{\Sscript{1}}\)};

\node at (2,-2.5) (a5) {\(a\)};
\node at (4,-2.5) (b5) {\(b\)};
\draw [->]  (a5) .. controls (2.5,-2.0) and (3.5,-2.0) ..  (b5) node[midway,above]{\(f_{\Sscript{1}}\)};
\draw [->]  (a5) .. controls (2.5,-3.0) and (3.5,-3.0) ..  (b5) node[midway,below]{\(f_{\Sscript{2}}\)};

\node at (6,-2.5) (a6) {\(a\)};
\node at (8,-2.5) (b6) {\(b\)};
\draw [->]  (a6) .. controls (6.5,-2.0) and (7.5,-2.0) ..  (b6) node[midway,above]{\(f_{\Sscript{1}}\)};
\draw [->]  (a6) .. controls (6.5,-3.0) and (7.5,-3.0) ..  (b6) node[midway,below]{\(f_{\Sscript{2}}\)};
\node at (8,-3.5) (c6) {\(c\)};

\node at (10,-2.5) (a7) {\(a\)};
\node at (12,-2.5) (b7) {\(b\)};
\draw [->]  (a7) .. controls (10.5,-2.0) and (11.5,-2.0) ..  (b7) node[midway,above]{\(f_{\Sscript{1}}\)};
\draw [->]  (a7) .. controls (10.5,-3.0) and (11.5,-3.0) ..  (b7) node[midway,below]{\(f_{\Sscript{2}}\)};
\node at (10,-3.5) (c7) {\(c\)};
\draw [->]  (a7) -- (c7) node[midway,left]{\(f_{\Sscript{3}}\)};

\draw[draw=black, dashed] (-0.5,0.5) rectangle (0.5,-0.5);
\draw[draw=black, dashed] (1.5,0.75) rectangle (4.5,-0.5);
\draw[draw=black, dashed] (5.5,1.5) rectangle (8.5,-0.5);
\draw[draw=black, dashed] (9.25,1.5) rectangle (12.5,-0.75);
\draw[draw=black, dashed] (1.5,-1.25) rectangle (4.5,-3.75);
\draw[draw=black, dashed] (5.5,-1.25) rectangle (8.5,-4.0);
\draw[draw=black, dashed] (9.25,-1.25) rectangle (12.5,-4.0);
\end{tikzpicture}
\end{center}
are all examples of not weakly equivalent  categorified \(G\)-sets, thus they give rise to different elements in \(\siL(G)\).
For example, consider \(G=\one\): in this case the \(G\)-action is trivial, thus \(\lL(G)=\mathbb{N}\) (we just have finite sets) but, regarding \(\siL(G)\), we have to consider all the classes given by all the previous not weakly equivalent  categorified \(G\)-sets.
More specifically, for each \(n \in \mathbb{N}^+\), in the case of \(\lL(\one)\), we just have \(n\) points, but in the case of \(\siL(G)\), we have to consider all the possible graphs with \(n\) vertices!
\end{example}

\section{Categorified Burnside ring of a groupoid}\label{sec:Grpds}
In this section we will analyse the  situation for a given groupoid. More precisely,  we will introduce and examine the category of categorified $\cG$-sets for a given groupoid $\cG$, and prove analogues results concerning its associated \emph{categorified Burnside ring}. We also provide a non isomorphic map from the classical Burnside ring to the categorified one. The necessary definitions can be found in~\cite{KaoKowMoritaHAlgd}, \cite{Kaoutit/Spinosa:2018} and~\cite{KaoutitSpinosaBurn}: here we will recall only the essential notions. 

\begin{definition}
A \imp{groupoid} is a (small) category such that all its morphisms are invertible. Given a groupoid \(\mathcal{G}\), a set \(X\) and a map \(\varsigma \colon X \longrightarrow \mathcal{G}_{0}\), we say that \(\left({X,\varsigma}\right)\) is a \textit{right \(\Cat{G}\)-set}, with a \imp{structure map} \(\varsigma\), if there is a \imp{right action} \(\rho \colon X  \pdue{\times}{\varsigma}{\mathsf{t} } \Cat{G}_{1} \longrightarrow X\), sending \(\left({x, g}\right)\) to \(xg\), and satisfying the following conditions:
\begin{enumerate}
\item for each \(x \in X\) and \(g \in \mathcal{G}_{1}\) such that \(\varsigma\left({x}\right)=\mathsf{t}\left({g}\right)\), we have \(\mathsf{s}\left({g}\right)=\varsigma\left({xg}\right)\);
\item for each \(x \in X\), we have \(x \iota_{\varsigma\left({x}\right) } = x\);
\item for each \(x \in X\) and \(g, h \in \mathcal{G}_{1}\) such that \(\varsigma\left({x}\right)= \mathsf{t}\left({g}\right)\) and \(\mathsf{s}\left({g}\right)=\mathsf{t}\left({h}\right)\), we have \(\left({xg}\right)h = x\left({gh}\right)\).
\end{enumerate}
The pair $(X,\varsigma)$ is referred to as a \emph{right $\cG$-set}. If there is no confusion, then we will also say that $X$ is a \emph{right groupoid-set}.  
\end{definition}

The main difference in the groupoid situation is that every object has a structure map that rules the groupoid action over itself; this can be directly seen, once a  groupoid-set is realized as a functor from the underlying category of $\cG$ to the core of the category of sets\footnote{The core category of a given  category is the subcategory whose morphisms are all isomorphisms.}. From now on all groupoid-sets are right ones. 

\begin{definition}
Given a groupoid \(\Cat{G}\), a \imp{morphism of  \(\Cat{G}\)-sets}, also called a \imp{ \(\Cat{G}\)-equivariant map}, \(F \colon \left({X, \varsigma}\right) \longrightarrow \left({Y, \vartheta}\right)\) is a function \(F \colon X \longrightarrow Y\) such that the following diagrams commute:
\[ 
\begin{minipage}[c]{0.3\linewidth}
\centering
\xymatrix{
X \ar[rd]_{\varsigma}   \ar[rr]^{F} &  & Y \ar[dl]^{\vartheta} \\
& \Cat{G}\sped{0}   }
\end{minipage}\hfill
\begin{minipage}[c]{0.2\linewidth}
\centering
\text{and} 
\end{minipage}\hfill
\begin{minipage}[c]{0.3\linewidth}
\centering
\xymatrix{
X \pdue{\times}{\varsigma}{\mathsf{t} } \Cat{G}\sped{1}   \ar[rr]^{ } \ar[d]_{F \times \Id_{\Cat{G}\sped{1} } } & & X  \ar[d]^{F}  \\
Y \pdue{\times}{\vartheta}{\mathsf{t} } \Cat{G}\sped{1}  \ar[rr]^{ } & & Y.
}
\end{minipage}
\]
\end{definition}

The resulting category, denoted by \(\Rset{G}\), has two monoidal structures: the disjoint union \(\Uplus\) and the fibre product, defined as follows.
Given \((X, \varsigma), (Y,\vartheta) \in \Rset{G}\), we set:
\[
(X, \varsigma)\, \underset{\Sscript{\Cat{G}\sped{0}}}{\times} \, (Y,\vartheta) \,=\,  \big( X \underset{\Sscript{\Cat{G}\sped{0}}}{\times} Y, \varsigma\vartheta\big),
\]
where \(X \fpro{\Cat{G}\sped{0}} Y = X \pdue{\times}{\varsigma}{\vartheta} Y\) and $\varsigma\vartheta \colon X \underset{\Sscript{\Cat{G}\sped{0}}}{\times} Y \longrightarrow \Cat{G}\sped{0}$ sends \((x,y)\) to \(\varsigma(x)=\vartheta(y)\).

A categorified \(\Cat{G}\)-set \(\Xset{\varsigma}\) is defined as an internal category in \(\Rset{G}\) with set of objects \(\left({\Xob, \varsigma\sped{0}}\right)\) and set of morphisms \(\left({\Xmo, \varsigma\sped{1}}\right)\).
Furthermore, in the monoidal structure that realize the multiplication of the Burnside rig, the cartesian product \(\times\) has to be replaced by the fibre product \(\fpro{\Cat{G}\sped{0}}\).

\begin{proposition}
The object
\[ \emptyset = \Big( \left({\emptyset, \emptyset}\right), \left({\emptyset, \emptyset}\right), \emptyset, \emptyset, \emptyset, \emptyset \Big) ,
\]
is initial in \(\CRset{\Cat{G}}\), \(\Uplus\) is a coproduct in \(\CRset{\Cat{G}}\) and \(\left({\CRset{\Cat{G}}, \Uplus, \emptyset}\right)\) is a strict monoidal category.
\end{proposition}

As in the group case, we have an ``inclusion'' functor from the category of usual  \(\Cat{G}\)-sets to the category of  categorified \(\Cat{G}\)-sets:
\begin{equation}\label{Eq:IcG} 
\begin{aligned}{\SInc{\Cat{G}} }  \colon & {\Rset{G} } \longrightarrow {\CRset{\Cat{G}} } \\ & {\left({X, \varsigma}\right)  }  \longrightarrow {\Big( \left({X, \varsigma}\right), \left({X, \varsigma}\right), \mathsf{s}\ped{X }, \mathsf{t}\ped{X }, \iota\ped{X }, m\ped{X } \Big) }\end{aligned} 
\end{equation}
where \(\mathsf{s}\ped{X} = \mathsf{t}\ped{X} = \iota\ped{X} = \id{X}\) and
\[ \begin{aligned}{m\ped{X}= \pr\sped{1}  }  \colon & {\left({X, \varsigma}\right)\sped{2} } \longrightarrow {\left({X, \varsigma}\right) } \\ & {(a,b)}  \longrightarrow {a}\end{aligned}
\]
with \(a = \id{X}(a) = \mathsf{s}\ped{\Cat{X} }(a)= \mathsf{t}\ped{X}(b)= b\).
The behaviour of \(\SInc{G}\) on morphisms is obvious.
Basically, the image of \(\SInc{G}\) is given by discrete categories.
Moreover, we will use the abuse of notation \(\Cat{G}\sped{0} = \SInc{G}\left({\Cat{G}\sped{0}}\right)\).

\begin{proposition}
We have \(\left({\CRset{\Cat{G}}, \fpro{\Cat{G}\sped{0} }, \Cat{G}\sped{0}}\right)\) is a monoidal category: the associator on \(\CRset{\Cat{G}}\)
\[ \left({ \left({ - \fpro{\Cat{G}\sped{0} } - }\right) \fpro{\Cat{G}\sped{0} } - }\right) \longrightarrow \left({ - \fpro{\Cat{G}\sped{0} } \left({- \fpro{ \Cat{G}\sped{0} } -}\right) }\right) 
\]
is the identity and the natural isomorphisms
\[ \Phi \colon \left({ \id{\CRset{\Cat{G}} }  \fpro{\Cat{G}\sped{0} }   \Cat{G}\sped{0} }\right) \longrightarrow \id{\CRset{\Cat{G}} } ,
\qquad 
\Psi \colon \left({ \Cat{G}\sped{0}   \fpro{\Cat{G}\sped{0} }  \id{\CRset{\Cat{G}} }   \longrightarrow \id{\CRset{\Cat{G}} } }\right)
\]
are defined as follows.
Let be \(\Xset{\varsigma}\): for \(i=0,1\) we set
\[ \begin{aligned}{\Phi\left({\Cat{X} }\right)\sped{i} }  \colon & {\Xmoi \fpro{\Cat{G}\sped{0} } \Cat{G}\sped{0} } \longrightarrow {\Xmoi} \\ & {(a,b)}  \longrightarrow {a}\end{aligned} 
\qquad \text{and} \qquad
\begin{aligned}{\Psi\left({\Cat{X} }\right)\sped{i} }  \colon & {\Cat{G}\sped{0} \fpro{\Cat{G}\sped{0} } \Xmoi  } \longrightarrow {\Xmoi} \\ & {(b,a)}  \longrightarrow {a}\end{aligned} 
\]
where \(\varsigma_i(a)=b\).
\end{proposition}

Now we have two monoidal structures \(\left({\CRset{\Cat{G}}, \Uplus, \emptyset}\right)\) and \(\left({\CRset{\Cat{G}}, \fpro{\Cat{G}\sped{0} }, \Cat{G}\sped{0}}\right)\) and to our aims, it is necessary to show the distributivity of \(\fpro{\Cat{G}\sped{0}}\) over \(\Uplus\).
To this end, let us consider \(\Xset{\varsigma}, \Yset{\vartheta}, \Aset{\omega} \in \CRset{\Cat{G}}\): we have to construct a morphism
\[ \lambda \colon \left[{ \Xset{\varsigma} \Uplus \Yset{\vartheta} }\right] \fpro{\Cat{G}\sped{0} } \Aset{\omega} \longrightarrow \left[{ \Xset{\varsigma}   \fpro{\Cat{G}\sped{0} } \Aset{\omega} }\right] \Uplus \left[{ \Yset{\vartheta}  \fpro{\Cat{G}\sped{0} } \Aset{\omega} }\right] 
\]
in \(\CRset{G}\).
We define it as the couple of morphisms in \(\Rset{G}\), for \(i=0,1\),
\[ \lambda\sped{i} \colon \left[{\left({\Xmoi, \varsigma\sped{i} }\right) \Uplus \left({\Ymoi, \vartheta\sped{i} }\right) }\right] \fpro{\Cat{G}\sped{0} } \left({\Amoi, \omega\sped{i} }\right) \longrightarrow \left[{\left({\Xmoi, \varsigma\sped{i} }\right)   \fpro{\Cat{G}\sped{0} } \left({\Amoi, \omega\sped{i} }\right) }\right] \Uplus \left[{ \left({\Ymoi, \vartheta\sped{i} }\right)  \fpro{\Cat{G}\sped{0} } \left({\Amoi, \omega\sped{i} }\right) }\right]
\]
that send \((a,b)\) to \((a,b)\) both if \(\varsigma\sped{i}(a)= \omega\sped{i}(b)\) with \(a \in \Xmoi\) and if \(\vartheta\sped{i}(a)= \omega\sped{i}(b)\) with \(a \in \Ymoi\).
The proof that \(\lambda\) is actually a morphism in \(\CRset{\Cat{G}}\) is now obvious. 
Following~\cite[Appendices]{KaoutitSpinosaBurn}, the category $\CRset{\cG}$ is actually a Laplaza category. In this direction, one can then apply the general construction of the Burnside rig and Grothendieck functor.

\begin{remark}
In general a groupoid \(\Cat{G}\) cannot be a right categorified \(\Cat{G}\)-set, once it is viewed  as a pair of $\cG$-sets via the  $\cG$-sets $\left(\Ga,\Sf{t} \right)$ and $\left( \Go, id_{\Sscript{0}}\right)$. If this were the case, then the  source map $ \Sf{s} \colon  \left( \Ga,\Sf{t} \right) \to \left( \Go, id_{\Sscript{0}}\right)$ should be $\cG$-equivariant, which in particular implies $\Sf{s}=\Sf{t}$. That is, $\cG$ should be a bundle of groups. Conversely, any group bundle with a fixed section leads to a groupoid  whose source is equal to its target and, thus, to a categorified set of that form.
Actually, a general groupoid \(\Cat{G}\) can be seen as a ``twisted '' and ``asymmetrical'' version of a right categorified \(\Cat{G}\)-set: this idea will be explored in a forthcoming paper by the same authors.
%
%
%

\end{remark}

The following results are adaptations of the corresponding results from~\cite{KaoutitSpinosaBurn} to the new situation of categorified sets.

\begin{proposition}
\label{pCatSimplRestric}
Given a groupoid \(\Cat{G}\), let \(\Cat{A}\) be a subgroupoid of \(\Cat{G}\).
We define a functor
\[ F \colon \CRset{\Cat{G}} \longrightarrow \CRset{\Cat{A} } 
\]
in the following way: let be \(\Xset{\varsigma} \in \CRset{\Cat{G} }\).
We define \(F\left({\Xset{\varsigma}}\right)\) as the internal category in \(\Rset{A}\) with set of objects \(\left({\varsigma\sped{0}^{-1}\left({\Cat{A}\sped{0} }\right), \left.{\varsigma\sped{0} }\right|_{\varsigma\sped{0}^{-1}\left({\Cat{A}\sped{0} }\right)}}\right)\) and set of morphisms \(\left({\varsigma\sped{1}^{-1}\left({\Cat{A}\sped{1} }\right), \left.{\varsigma\sped{1} }\right|_{\varsigma\sped{1}^{-1}\left({\Cat{A}\sped{1} }\right)}}\right)\).
The source, target, identity and composition maps of \(F\left({\Xset{\varsigma}}\right)\) are the opportune restriction to \(\varsigma\sped{0}^{-1}\left({\Cat{A}\sped{0} }\right)\) and \(\varsigma\sped{1}^{-1}\left({\Cat{A}\sped{1} }\right)\) of the relative maps of \(\Xset{\varsigma}\).
Then \(F\) is a strict monoidal functor with respect to both \(\Uplus\) and the fibre product.
\end{proposition}
\begin{proof}
It proceeds as the proof of~\cite[Proposition 3.6]{KaoutitSpinosaBurn}.
\end{proof}

Given a groupoid $\Cat{G}$ and a fixed object $x \in \Cat{G}\sped{0}$, we denote with \(\Cat{G}^{\Sscript{(x)}}\) the one object subgroupoid with isotropy group \(\Cat{G}^{\Sscript{x}}\). That is, viewed as a groupoid \(\Cat{G}^{\Sscript{(x)}}\) has $\{x\}$ as a  set of objects and \(\Cat{G}^{\Sscript{x}}=\{ g \in \Cat{G}_{\Sscript{1}}|\,\, {\sf{s}}(g)={\sf{t}}(g)=x\}\) as a set of arrows. A groupoid $\Cat{G}$ with only one connected component (i.e., for any pair of objects $x, y \in \Ga $, there is an arrow $g \in \Ga $ such that ${\sf{s}}(g)=x$ and  ${\sf{t}}(g)=y$), is called  \emph{a transitive groupoid}.

\begin{theorem}
\label{tEquivTransIsoGrp}
Given a transitive  and not empty groupoid \(\mathcal{G}\), let be \(a \in \mathcal{G}_{ \ssstyle{0} }\).
Then there is an equivalence of monoidal categories with respect to both \(\Uplus\) and the fibre product, that is, 
\[ 
\CRset{\Cat{G}} \simeq  \CRset{\Cat{G}^{\Sscript{(a)}}} .
\]
\end{theorem}
\begin{proof}
It is in the same lines as that of  \cite[Theorem 3.9]{KaoutitSpinosaBurn}.
Let's set \(\Cat{A}\sped{0}= \Set{a}\). In one direction, we 
the functor \(F \colon \CRset{\Cat{G} } \longrightarrow \CRset{\Cat{A}}\), the one constructed in Proposition~\(\ref{pCatSimplRestric}\). In the other direction, the functor \(G \colon \CRset{\cA} \longrightarrow \CRset{\cG}\) is defined as follows:   Given an object $(\cX,\varsigma) $ in $\CRset{\cA}$, we set \(G\left({\Xset{\varsigma}}\right)\) as the internal category in \(\Rset{G}\) with set of objects
\[ \Big( \Yob = \Xob \times \Cat{G}\sped{0}, \, \vartheta\sped{0} = \pr\sped{2} \colon \Yob = \Xob \times \Cat{G}\sped{0} \longrightarrow \Cat{G}\sped{0}   \Big) 
\]
and object of morphisms
\[ \Big( \Ymo = \Xmo \times \Cat{G}\sped{0}, \, \vartheta\sped{1} = \pr\sped{2} \colon \Ymo = \Xmo \times \Cat{G}\sped{0} \longrightarrow \Cat{G}\sped{0}   \Big) .
\]
The source, target, identity and composition maps of \(G\left({\Xset{\varsigma}}\right)\) are defined as \(\mathsf{s}\ped{G\left({\Xset{\varsigma}}\right) } = \mathsf{s}\ped{\Xset{\varsigma} } \times \Cat{G}\sped{0}\), \(\mathsf{t}\ped{G\left({\Xset{\varsigma}}\right) } = \mathsf{t}\ped{\Xset{\varsigma} } \times \Cat{G}\sped{0}\), \(\iota\ped{G\left({\Xset{\varsigma}}\right) } = \iota\ped{\Xset{\varsigma} } \times \Cat{G}\sped{0}\) and \(m\ped{G\left({\Xset{\varsigma}}\right) }\Big((x,a), (y,b) \Big) = \Big( m\ped{\Xset{\varsigma} }(x,y), a\Big)\) for each \(\left({(x,a), (y,b)}\right) \in \Xmod\).
The functors $F$ and $G$ are mutualy inverse and induces the stated equivalence of symmetric monoidal categories. 
\end{proof}

\begin{proposition}
\label{pBurnRigCoprod}
The Burnside rig functor \(\siL\) sends coproduct to product.
In particular, given a family of groupoids  \(\left(\Cat{G}_j \right)_{j \, \in \,  I}\), let \(\left(i_j \colon \Cat{G}_j  \longrightarrow  \Cat{G}\right)_{j \, \in \,  I }\) be their coproduct in the category \(\Grps\) of groupoids.
Then
\[ \left(\siL\left({i_j}\right) \colon \siL\left({\Cat{G}}\right)  \longrightarrow \siL\left({\Cat{G}_j}\right) \right)_{j \in I} 
\]
is the product of the family \(\left(\siL\left({ \Cat{G}_j}\right)\right)_{j \, \in \,  I}\) in the category \(\riG\).
\end{proposition}
\begin{proof}
The idea of the proof is similar to \cite[Prop.~6.3]{KaoutitSpinosaBurn}: the only difference to keep in mind is that we are dealing with particular small categories and not sets.
\end{proof}

\begin{theorem}
\label{tEquivCatIsomBurn}
Let be \(\Cat{G}\) and \(\Cat{A}\) be groupoids such that there is a symmetric strong monoidal  equivalence of categories \(\CRset{\Cat{G}} \simeq \CRset{\Cat{A}}\) with respect to both \(\Uplus\) and the fibre product.
Then there is an isomorphism of commutative rings \(\siB\left({\mathcal{G} }\right) \cong \siB\left({\mathcal{A} }\right)\).
\end{theorem}
\begin{proof}
It proceeds as in~\cite[Thm.~6.5]{KaoutitSpinosaBurn}.
\end{proof}

\begin{theorem}
\label{tBurnProdComp}
Given a groupoid \(\Cat{G}\), fix a set of representative objects $\rep(\Cat{G}\sped{0})$ representing  the set of connected components $\pi_{\Sscript{0}}(\Cat{G})$.
For each \(a \, \in \, \rep(\Cat{G}\sped{0})  \), let  \(\Cat{G}^{\Braket{a}}\) be the connected component of \(\Cat{G}\) containing $a$, which we consider as a groupoid.
Then we have the following isomorphism of rings:
\[ 
\siB\left({\mathcal{G}}\right) \cong \prod_{a \, \in \, \rep(\Cat{G}\sped{0}) } \siB\left({ \Cat{G}^{\Braket{a}}  }\right).
\]
\end{theorem}
\begin{proof}
It follows directly from Proposition~\(\ref{pBurnRigCoprod}\).
\end{proof}

\begin{corollary}
Given a groupoid \(\mathcal{G}\), we have the following, non canonical,  isomorphism of rings:
\[ \siB\left({\mathcal{G}}\right) \cong \prod_{a \, \in \, \rep(\Cat{G}\sped{0})} \siB\left({\mathcal{G}^{ \ssstyle{a} }  }\right),
\]
where the right hand side term is the product of commutative rings. 
\end{corollary}
\begin{proof}
Immediate from Theorem~\(\ref{tBurnProdComp}\), Theorem~\(\ref{tEquivCatIsomBurn}\), and Theorem~\(\ref{tEquivTransIsoGrp}\).
\end{proof}

Lastly, as in the case of groups, we have a non monomorphism of commutative rings  from the classical Burnside ring of $\cG$ and its categorified Burnside ring. 
More precisely, we have a commutative diagram of rings:
$$
\xymatrix{ \bB(\cG) \ar@{->}^-{}[rr]   \ar@{->}_-{\cong}[d]  & & \siB(\cG)  \ar@{->}_-{\cong}[d]  \\   \underset{a \, \in \, \rep(\Cat{G}\sped{0})}{\prod} \bB\left({\mathcal{G}^{ \ssstyle{a} }  }\right)  \ar@{->}^-{}[rr] & &   \underset{a \, \in \, \rep(\Cat{G}\sped{0})}{\prod} \siB\left({\mathcal{G}^{ \ssstyle{a} }  }\right) }
$$
where the vertical morphism of rings, are deduced by using the functor of equation \eqref{eIncFun} and \eqref{Eq:IcG}.

\printbibliography

\end{document}